\documentclass[12pt]{article}
\usepackage{setspace} 
\usepackage{epsfig}
\usepackage{subfig}
\usepackage[titletoc,title]{appendix}
\usepackage{amsmath,amsthm,amsfonts}
\usepackage[bottom]{footmisc}
\RequirePackage[colorlinks,citecolor=blue,urlcolor=blue,linkcolor=blue]{hyperref}
\hypersetup{
colorlinks = true,
citecolor=blue,
urlcolor=blue,
linkcolor=blue,
pdfauthor = {Alexey Kuznetsov},
pdfkeywords = {},
pdftitle = {On series expansions of zeros of the deformed exponential function},
pdfpagemode = UseNone
}

    \def\qed{\hfill\(\sqcap\kern-8.0pt\hbox{\(\sqcup\)}\)\\}

    \def\c{{\mathbb C}}
    \def\d{{\textnormal d}}

	\newtheorem{theorem}{Theorem}
	\newtheorem{lemma}{Lemma}
	\newtheorem{proposition}{Proposition}

	\newtheorem{remark}{Remark}
	\renewcommand{\footnote}{\alph{footnote}}
\title{On series expansions of zeros of the deformed exponential function}
\author{ 
{Alexey Kuznetsov\footnote{Dept. of Mathematics and Statistics,  York University,
4700 Keele Street, Toronto, ON, M3J 1P3, Canada.  
  Email: akuznets@yorku.ca}} }

\date{\today}

\begin{document}
%**************************************************************************************************
%**************************************************************************************************
%**************************************************************************************************
\maketitle

\begin{abstract}
For \( q \in (0, 1) \), the deformed exponential function 
\(
f(x) = \sum_{n \geq 1} x^n q^{n(n-1)/2}/n!\)
is known to have infinitely many simple and negative zeros \( \{x_k(q)\}_{k \geq 1}\). In this paper, we analyze the series expansions of \( -x_k(q)/k \) and \( k/x_k(q) \) in powers of \( q \). We prove that the coefficients of these expansions are rational functions of the form \( P_n(k)/Q_n(k) \) and \( \widehat{P}_n(k)/Q_n(k) \), where \( Q_n(k)  \in {\mathbb Z}[k]\) is explicitly defined and the polynomials \( P_n(k),  \widehat{P}_n(k)\in {\mathbb Z}[k] \) can be computed recursively. We provide explicit formulas for the leading coefficients of \( P_n(k) \) and \( \widehat{P}_n(k) \) and compute the coefficients of these polynomials for \( n \leq 300 \). Numerical verification shows that \( P_n(k) \) and \( \widehat{P}_n(k) \) take non-negative values for all \( k \in \mathbb{N} \) and \(n\le 300\), offering further evidence in support of conjectures by Alan Sokal. 
\end{abstract}
{\vskip 0.15cm}
{\vskip 0.25cm}

%*******************************************************************************
%*******************************************************************************
%*******************************************************************************
\section{Introduction and the main results}\label{section_intro}
%*******************************************************************************
%*******************************************************************************
%*******************************************************************************

The deformed exponential function is defined for \(x, q \in {\mathbb C}\) with \(|q| \le  1\) as
\begin{equation}\label{eq:def_f(x,q)}
f(x)=f(x;q):=\sum\limits_{n\ge 0} \frac{x^n}{n!} q^{n(n-1)/2}. 
\end{equation}
It is the unique differentiable solution to the functional-differential equation
\[
f'(x)=f(qx), \;\;\; f(0)=1.
\] 
The deformed exponential function has significant applications in combinatorics \cite{Gessel_Sagan_1996,Mallows_1968, Scott_Sokal_2009} and statistical physics \cite{Scott_Sokal_2005}.
 When \(|q|=1\),  the function \(f(x)\) is entire of order one and type one and its properties have been extensively studied (see \cite{Eremenko_2007} and references therein). In this paper, we will focus on the case \(|q|<1\), where \(f\) is entire of order zero and admits the Hadamard product representation 
\begin{equation}\label{eqn_Hadamard}
f(x;q)=\prod\limits_{k\ge 1} \Big( 1- \frac{x}{x_k(q)}\Big).
\end{equation}
For $q\in (0,1)$, the zeros $x_k=x_k(q)$ of the deformed exponential function are known to be simple and negative, see \cite{Iserles_1993,Liu_1998,Morris_1972}. We assume that these zeros are arranged in the decreasing order, that is 
\( \dots <x_{k+1} < x_k<\dots<x_1<0\). 
 Note that we follow Wang and Zhang \cite{Wang_2018} and index the zeros by \(k=1,2,3,\dots\), whereas Langley \cite{Langley_2000}, Liu \cite{Liu_1998} and Sokal \cite{Sokal_2009} start indexing at \(k=0\). The implications of this choice are discussed in Section \ref{section_numerics}. 

The logarithm of the deformed exponential function is related to the generating polynomials of complete graphs \(C_n(v)\) (see \cite{Scott_Sokal_2009,Sokal_2009}) via the identity
\begin{equation}\label{ln_f_xq}
\sum\limits_{n\ge 1} \frac{x^n}{n!} C_n(v)= 
\ln \big(f(x;1+v)\big),  \;\;\; x\in \c, \; |1+v|<1. 
\end{equation}
Using \eqref{ln_f_xq} and the Hadamard product representation \eqref{eqn_Hadamard}, Sokal \cite{Sokal_2009} derived  an identity 
\begin{equation}\label{def_S_m}
\overline C_n(q):=C_n(q-1)=-(n-1)!\sum\limits_{k\ge 0} x_k(q)^{-n}, \;\;\; n\ge 1, \; q \in (0,1). 
\end{equation}
The polynomials \(\overline C_n(q)\) are connected to Tutte polynomials \cite{Scott_Sokal_2009,Gessel_Sagan_1996} and  enumerations of trees \cite{Mallows_1968}. They are known to satisfy the recurrence relation
\begin{equation}\label{eq:A_n_recurrence}
\overline C_{n+1}(q)=q^{n(n+1)/2}-
 \sum\limits_{j=1}^n \binom{n}{j} \overline C_{n+1-j}(q) 
 q^{j(j-1)/2}, 
\end{equation}
which can be established from \eqref{ln_f_xq} using logarithmic differentiation. 
The first three polynomials are given explicitly as
\begin{equation}\label{eq:S_123_explicit}
\overline C_1(q)=1, \;\; \overline C_2(q)=q-1, \;\; \overline C_3(q)=(q-1)^2(q+2). 
\end{equation}

Since the zeros of \(f(x;q)\) are simple for \(q\in (0,1)\), the functions \(x_k(q)\) are analytic in an open domain that contains \((0,1)\) (this follows from the Implicit Function Theorem). Sokal \cite{Sokal_2009,Sokal_2016} studied the analyticity of \(x_k(q)\) in the disk \(|q|<1\) and made several conjectures regarding their series expansions in powers of \(q\).  Throughout this paper, we denote \(D_t:=\{z \in \c \; : \; |z|<t\}\).

\begin{theorem}[Sokal, \cite{Sokal_2016}]\label{thm_Sokal}
\({}\)
\begin{itemize}
\item[(a)] \(x_1(q)\) is analytic in \(D_{t_1}\) with \(t_1=0.44175\) and satisfies 
\[
 1.25 |q|^{1/2} < |x_1(q)| < 1.26 |q|^{-1/2}, \;\;\; q\in
 D_{t_1};    
\]  
\item[(b)] \(x_2(q)\) is analytic in \(D_{t_2} \setminus \{0\}\) with \(t_2=0.31499\) and satisfies 
\[
 1.26 |q|^{-1/2} < |x_2(q)| < 2.38 |q|^{-3/2}, \;\;\; q\in
 D_{t_2} \setminus \{0\};    
\]   
\item[(c)] \(x_3(q)\) is analytic in 
\(D_{t_3} \setminus \{0\}\) with \(t_3=0.27814\) and satisfies 
\[
2.38 |q|^{-3/2} < |x_3(q)|<3.414 |q|^{-5/2}, \;\;\; 
q\in
 D_{t_3} \setminus \{0\};    
\]
\item[(d)] For \(q\in D_{t^*}\setminus\{0\}\) with \(t^*=0.207875\), all zeros \(\{x_k(q)\}_{k\ge 1}\) of \(f(x;q)\) are simple, analytic  and satisfy 
\[
0< |x_1| < |q|^{-1/2} < |x_2| < 2|q|^{-3/2} < |x_3| < 3 q^{-5/2} < |x_4|<4 q^{-7/2}<\dots. 
\]
\end{itemize}
\end{theorem}

As the paper \cite{Sokal_2016} is not easily accessible, we provide a proof of Theorem \ref{thm_Sokal} in Section \ref{section_proofs}. This result implies that we can expand \(x_k(q)\) in Laurent series 
 \[
  x_k(q)=\sum\limits_{n\ge 1-k} c_{k,n} q^n, \;\;\;
  0<|q|<t^*,
 \]
and Sokal \cite{Sokal_2009b,Sokal_2009,Sokal_2024} showed that the first coefficient in this series is \(c_{k,1-k}=-k\). 
Consequently, there exist coefficients \(a_{k,n}\) and \(\hat a_{k,n}\) such that for \(k\in {\mathbb N}\) and \(q \in D_{t^*} \setminus \{0\}\) we have 
\begin{equation}\label{x_k_series}
x_k(q)=- k q^{1-k} \Big( 1+ \sum\limits_{n\ge 1} a_{k,n} q^n \Big)= 
 \frac{-k q^{1-k}}{1 - \sum\limits_{n\ge 1} \hat a_{k,n} q^n }.
\end{equation} 
Sokal \cite{Sokal_2009, Sokal_2024} conjectured several properties of these series representations: 
\begin{itemize} \label{Sokal_conjectures}
\item {\bf Conjecture 1}: The two power series in 
\eqref{x_k_series} converge for all \(|q|<1\).
\item {\bf Conjecture 2}: \(a_{k,n}\ge 0\) for all  \(k,n\ge 1\).
\item {\bf Conjecture 3}: \(\hat a_{k,n} \ge 0\) for all \(k,n \ge 1\).
\end{itemize}
Conjectures 2 and 3 were verified numerically in \cite{Sokal_2009, Sokal_2024} for \(k=1\) and \(n\le 65535\)
 and for some other values of \(k\) and \(n\). It is easy to see that Conjecture 3 implies Conjecture 2, which in turn implies Conjecture 1. 

Recent paper by Wang and Zhang \cite{Wang_2018}, building on earlier result of Zhang \cite{Zhang_2016}, provides a very interesting asymptotic expansion of \(x_k\) in powers of \(k^{-1}\). Let us define
 \begin{equation}\label{def_Aj} 
 A_j(q):=\sum_{n \ge 1} n^j \sigma(n) q^n, \;\;\; j\ge 0, \; |q|<1,
 \end{equation} 
where \(\sigma(n)=\sum_{d|n} d\) is the sum of divisors function.
Wang and Zhang  \cite{Wang_2018} proved that for \(q \in (0,1)\) and any \(m\in {\mathbb N}\) the following asymptotic expansion holds as \(k\to +\infty\)
  \begin{equation}\label{Wang_Zhang_result}
 x_k(q)=-k q^{1-k} \Big( 1 + \sum_{i=1}^m B_i(q)k^{-1-i} + o(k^{-1-m}) \Big), 
 \end{equation}
 where 
\begin{equation}\label{formulas_Ci} 
 B_1(q)=A_0(q), \;\; B_2(q)=-A_1(q), \;\; B_3(q)=-\frac{1}{10}A_0(q)+\frac{3}{5} A_1(q)+\frac{1}{2} A_2(q)-
 \frac{13}{10} A_0(q)^2,
 \end{equation}
 and higher order coefficients \(B_i(q)\) are polynomials in \(A_0(q)\), \(A_1(q)\), \(A_2(q)\) with rational coefficients, which can be computed recursively. Note that the coefficients \(B_i(q)\) were denoted by \(C_i(q)\) in \cite{Wang_2018}. We adopt the notation \(B_i(q)\) to avoid confusion with the generating polynomials of complete graphs  discussed earlier.

Our first result demonstrates that, in addition to the results stated in \eqref{def_S_m}-\eqref{eq:A_n_recurrence}, there exist other infinite series/products involving the zeros \(x_k\) of the deformed exponential function that can be evaluated explicitly.  
\begin{theorem}\label{thm_main1}
For \(k\ge 1\) and \(q \in (0,1)\) the following three identities hold:
\begin{align}\label{eq:x_k_prime}
&\frac{\d x_k}{\d q}=\frac{x_k^2}{2} 
\sum\limits_{j\ge 1} \frac{1}{x_j-q x_k},\\
\label{eq:x_k_sum}
 &\sum\limits_{j\ge 1} \frac{1}{x_k-qx_j}=0,\\
 \label{eq:x_k_product}
 &(q-1) x_k = \prod\limits_{\stackrel{j\ge 1}{j\neq k}}
\frac{x_j-x_k}{x_j -qx_k}.
\end{align}
\end{theorem}

To state our second result in this paper, we define
\begin{equation}\label{def_gamma_nk_si}
\gamma_{n,l}:=\Big\lfloor \frac{2n}{l(l+1)} \Big\rfloor, 
\;\;\;\qquad Q_n(x):=x^n \prod\limits_{l\ge 1} (x+l)^{\gamma_{n,l}}, \;\;\;\qquad
\mu_j(n):=\sum\limits_{l\ge 1} l^j \gamma_{n,l}. 
\end{equation}
Note that \(\gamma_{n,l}=0\) for \(l\ge \sqrt{2n}\) and \(Q_n(x) \in {\mathbb Z}[x]\) is a polynomial  of degree
\begin{equation}\label{def_M_n}
M_n:={\textnormal{deg}}(Q_n)=\mu_0(n)+n. 
\end{equation}

\begin{theorem}\label{thm_main}\({}\)
\begin{itemize}
\item[(i)] There exist polynomials \(\{P_n\}_{n\ge 1}\) and \(\{\widehat P_n\}_{n\ge 1}\) with integer coefficients such that for all \\ \(k,n \ge 1\)
\begin{equation}\label{eqn_akn_hat_akn}
a_{k,n}=\frac{P_n(k)}{Q_n(k)}, \;\;\; 
\hat a_{k,n}=\frac{\widehat P_n(k)}{Q_n(k)}.
\end{equation} 
\item[(ii)]
We have \(M_n < 3n\) and \(n^{-1} M_n \to 3\) as \(n\to \infty\). For \(n\ge 2\) the polynomials \(P_n\) and \(\widehat P_n\) have the form  
\begin{equation}\label{form_Pn}
\sigma(n)k^{M_n-2} + \sigma(n) 
\big( \mu_1(n)-n  \big)  k^{M_n-3} + 
{\textnormal{lower order terms}}.
\end{equation}
The second coefficient in the above expression is zero for \(n=2\) and is positive for all \(n\ge 3\). 
\end{itemize}
\end{theorem}

The third, fourth, and other coefficients of \(P_n\) and \(\widehat P_n\) can also be computed explicitly, although the resulting expressions are complicated. For example, the coefficient in front of \(k^{M_n-4}\) of polynomial \(P_n\) (respectively, \(\widehat P_n\)) is 
\begin{equation}\label{coeff3_Pn}
\frac{\sigma(n)}{10}\big(
5n^2+6n-1+5 \mu_1(n)^2-5 \mu_2(n)- 10 n \mu_1(n)
\big)
-\frac{N}{10} 
\sum\limits_{j=1}^{n-1} \sigma(j)\sigma(n-j),
\end{equation}
where \(N=13\) (respectively, \(N=23\)). Note that the convolution sum in \eqref{coeff3_Pn} can be expressed in terms of \(\sigma(n)\) and  \(\sigma_3(n):=\sum_{d|n} d^3\), see \cite{Ramanujan_1916} or  \cite{Kim_2015}[formula (1.4)]. A sketch of the proof of \eqref{coeff3_Pn} is provided on page \pageref{page_sketch_proof}.

This paper is organized as follows. In Section \ref{section_proofs} we present the proofs of Theorems \ref{thm_Sokal}, \ref{thm_main1} and \ref{thm_main}. Section \ref{section_numerics} focuses on the numerical computation of polynomials \(P_n\) and \(\widehat P_n\) and on the analysis of their coefficients.  Specifically, we describe the computation of these polynomials for all \(n \le 300\) and verify that \(P_n(k)\ge 0\) and \(\widehat P_n(k)\ge 0\) for all \(n \le 300\) and all \(k\in {\mathbb N}\). Our findings confirm that \(a_{n,k}\) and \(\hat a_{n,k}\) are non-negative for all \(n\le 300\) and all \(k\ge 1\),  providing additional evidence in support of Sokal's conjectures.  Furthermore, we examine the patterns of positive and negative coefficients of \(P_n\) and \(\widehat P_n\), and we investigate the zeros of these polynomials in the interval $(1,\infty)$

%**************************************************************************
%**************************************************************************
%**************************************************************************

 \section{Proofs}\label{section_proofs}

 %**************************************************************************
 %**************************************************************************
 %**************************************************************************

The following proof of Theorem \ref{thm_Sokal} 
is a rephrasing of the original proof presented in \cite{Sokal_2016}.  Recall that \(D_t:=\{z \in \c \; : \; |z|<t\}\). We begin by  establishing the following auxiliary result.

\begin{lemma}\label{lemma_Sokal}
Assume that for some \(k\in {\mathbb N} \cup \{0\}\),  \(u>0\) and \(t\in (0,1)\)
\begin{equation}\label{def_G_k}
 G_k(u;t):=k! \sum\limits_{n\ge 0} \frac{u^{n-k}}{n!} t^{(n-k)^2/2}\le 2. 
\end{equation}
 Then for any \(q\in D_t \setminus\{0\}\) the function \(x\mapsto f(x;q)\) has exactly \(k\) zeros inside the disk \(D_{r}\) with \(r=u |q|^{1/2-k} \) and no zeros on its boundary.  
\end{lemma}
\begin{proof}
Fix \(k\in {\mathbb N}\cup\{0\}\). Assume that for some \(q \in D_1 \setminus\{0\}\) and for all \(x\) on the circle \(|x|=\rho>0\), we have 
\begin{equation*}
\Big | f(x;q)- 2\frac{x^k}{k!} q^{k(k-1)/2}\Big|< 2\frac{\rho^k}{k!} |q|^{k(k-1)/2}. 
\end{equation*}
Then, by Rouche's Theorem, the function \(f(x;q)\) has exactly \(k\) zeros inside \(D_{\rho}\), and none on the circle \(|x|=\rho\). 
Next, we derive an upper bound  
\begin{equation*}
\Big | f(x;q)- 2\frac{x^k}{k!} q^{k(k-1)/2}\Big| 
= \Big |  \sum\limits_{n\ge 0} (-1)^{\delta_{n,k}} \frac{x^n}{n!} q^{n(n-1)/2}\Big| \le 
\sum\limits_{n\ge 0}  \frac{|x|^n}{n!} |q|^{n(n-1)/2}.
\end{equation*}
Here \(\delta_{n,k}=1\) if \(n=k\) and \(\delta_{n,k}=0\) otherwise. 
Combining these results, we see that if for some \(\rho>0\) and \(q \in D_1 \setminus\{0\}\) 
\begin{equation}\label{prop_roots_1}
\sum\limits_{n\ge 0}  \frac{\rho^n}{n!} |q|^{n(n-1)/2} < 
2\frac{\rho^k}{k!} |q|^{k(k-1)/2},
\end{equation} 
then the function \(f(x;q)\) has \(k\) roots inside \(D_{\rho}\).
 Now we set \(|q|=t\), \(\rho=u t^{1/2-k}\), divide both sides of \eqref{prop_roots_1} by the right-hand side and, after simplifying the resulting left-hand side, we obtain the following result:  the inequality \eqref{prop_roots_1} (with \(|q|=t\) and \(\rho=u t^{1/2-k}\)) is equivalent to \(G_k(u;t)<2\). The function \(G_k(u,t)\)
   is manifestly increasing in \(t\). Thus, if \(G_k(u,t) \le 2\), then also
   \(G_k(u,|q|) < 2\) whenever \(|q| < t\).  This completes the proof of Lemma 1.
\end{proof}

\noindent
{\bf Proof of Theorem  \ref{thm_Sokal}:}
Let \(t_1=0.44175\). 
We verify numerically that 
\[G_0(1.25; t_1)=1.99164\cdots<2.\] By Lemma \ref{lemma_Sokal},  for \(q\in D_{t_1}\setminus \{0\}\), the function \(f(x;q)\) has no zeros in the disk \(|x|\le 1.25|q|^{1/2}\).Similarly, we verify that 
 \[G_1(1.26;t_1)=1.99999424\cdots<2.\] By Lemma \ref{lemma_Sokal}, for \(q \in D_{t_1}\setminus \{0\}\) the function \(x\mapsto f(x;q)\) has exactly one zero \(x_1(q)\) which satisfies 
\(|x_1(q)|< 1.26 |q|^{-1/2}\) (and it has no zeros on the boundary of this disk). Moreover, \(x_1(q)\) is analytic in \(D_{t_1}\) (note that \(x_1(0)=-1\), thus \(q=0\) is a removable singularity for \(x_1(q)\)). Combining the above facts gives us the result in item (i) of Theorem \ref{thm_Sokal}. 

We follow the same approach for other values of \(k\). We  numerically verify that \(G_2(2.38; t_2)<2\), where \(t_2=0.31499\). According to Lemma \ref{lemma_Sokal}, for every \(q \in D_{t_2} \setminus \{0\}\) the function \(f(x;q)\) has exactly two zeros \(x_1(q)\) and \(x_2(q)\) in the disk \(|x|<2.38 |q|^{-3/2}\). Given the results about \(x_1(q)\) established earlier and the fact that there are no zeros satisfying \(|x|=1.26 |q|^{-1/2}\), we conclude that \(x_2(q)\)  lies in the annulus
\[
 1.26 |q|^{-1/2} < |x_2| < 2.38 |q|^{-3/2} 
\]   
 for \(q\in D_{t_2} \setminus \{0\}\). The proof of Theorem \ref{thm_Sokal}(c) follows similarly by verifying that $G_3(3.414;t_3)<2$, where \(t_3=0.27814\), and applying Lemma \ref{lemma_Sokal}.  

To prove item (d), we use the inequality \(k! k^{n-k} \le  n! \)
for all \(k\ge 1\) and \(n\ge 0\) and check that for \(t\in (0,1)\)
\[
G_k(kt;1)\le \sum\limits_{n\ge 0} t^{(n-k)^2/2} =
\sum\limits_{m\ge -k} t^{m^2/2}< \sum\limits_{m \in {\mathbb Z}} t^{m^2/2}.
\]
For $t= t^*=0.207875$ we compute 
\[
\sum\limits_{m \in {\mathbb Z}} (t^*)^{m^2/2}=1.9999999368\cdots<2.
\]
Thus, $G_k(kt;1) < 2$ for all \(k\in {\mathbb N}\) and \(0<t\le t^*\).
 Applying Lemma \ref{lemma_Sokal} gives the result in item (d).  
\qed

\noindent
{\bf Proof of Theorem  \ref{thm_main1}:}
The proof relies on the identities
\begin{equation}\label{eq:f_x_f_q}
f_x(x;q)=\frac{\partial}{\partial x} f(x;q)=f(qx;q), \;\;\; f_q(x;q)=\frac{\partial}{\partial q} f(x;q)=\frac{1}{2} x^2 f(q^2 x;q), \;\;\; x\in \c, \;\; |q|<1, 
\end{equation}
which can be derived directly from the definition \eqref{eq:def_f(x,q)}.  From \eqref{eq:f_x_f_q} and the Hadamard product factorization \eqref{eqn_Hadamard}  we find:  
\begin{equation}\label{eq:logarithmic_derivative}
\frac{f(qx;q)}{f(x;q)}=\frac{f_x(x;q)}{f(x;q)}=\frac{\partial}{\partial x} \ln (f(x;q))= \sum\limits_{j\ge 1} \frac{1}{x-x_j}.
\end{equation}
Next we set  \(x=x_k/q\) in \eqref{eq:logarithmic_derivative} and obtain 
\eqref{eq:x_k_sum} using the facts that \(f(x_k;q)=0\) and \(f(x_k/q;q) \neq 0\). If the latter condition was false, then \(f\) would have a double zero at \(x_k/q\), which is impossible, since all zeros $x_k$ are simple for \(q\in (0,1)\).

Next, taking derivative in \(q\) of both sides of the identity \(f(x_k(q); q)=0\), we obtain
\[
 f_x(x_k(q);q) x_k'(q)+ f_q(x_k(q);q)=0. 
\]
Using the above equation and \eqref{eq:f_x_f_q} and \eqref{eq:logarithmic_derivative} we compute
\[
x_k'=-\frac{f_q(x_k;q)}{f_x(x_k;q)}=-\frac{x_k^2}{2} \times \frac{f(q^2 x_k;q)}{f(q x_k;q)}=-\frac{x_k^2}{2} 
\sum\limits_{j\ge 1} \frac{1}{x_k q - x_j}.
\]
This establishes \eqref{eq:x_k_prime}. 

To prove \eqref{eq:x_k_product}, we  compute \(f_x(x_k;q)\) in two ways. Using \eqref{eqn_Hadamard} and the fact that \(f(x_k;q)=0\), we have
\begin{align*}
 f_x(x_k;q)&=\lim_{h\to 0} \frac{1}{h} f(x_k+h;q)=
  \lim_{h\to 0} \frac{1}{h} \prod\limits_{j\ge 1} \Big( 1 - \frac{x_k+h}{x_j} \Big)
  \\
  &=
   \lim_{h\to 0} \frac{1}{h} \Big( 1 - \frac{x_k+h}{x_k} \Big) 
   \times  \prod\limits_{\stackrel{j\ge 1}{j\neq k}} \Big( 1 - \frac{x_k+h}{x_j} \Big)
   = -\frac{1}{x_k} \prod\limits_{\stackrel{j\ge 1}{j\neq k}} \Big( 1 - \frac{x_k}{x_j} \Big). 
\end{align*}
On the other hand, using \eqref{eq:f_x_f_q}, we  conclude that
\[
f_x(x_k;q)=f(q x_k ; q )=\prod\limits_{j\ge 1} \Big( 1 - \frac{q x_k}{x_j} \Big)
=(1-q) \prod\limits_{\stackrel{j\ge 1}{j\neq k}} \Big( 1 - \frac{q x_k}{x_j} \Big).
\]
Equating the two expressions for \(f_x(x_k;q)\) gives us the desired result \eqref{eq:x_k_product}.
\qed

Before proving Theorem \ref{thm_main}, we establish two lemmas. 
For \(i\ge 1\) and \(k\neq 0\), define
\begin{equation}\label{def_alpha_i}
\alpha_i(k)=\prod\limits_{l=1}^i \Big(1-\frac{l}{k} \Big),
\end{equation}
and for \(w \in \c \setminus \{0\}\) and \(|q|<1\), define 
\begin{equation}\label{def_F_k_wq}
F_k(w;q):=\sum\limits_{i\ge 1} (-1)^i 
\Big[ \alpha_i(k) w^{-i} -  \alpha_i(-k)^{-1} w^{i+1} \Big]q^{i (i-1)/2}.
\end{equation}
We also introduce functions \(w_k(q)\) via  
\(x_k(q)=- k q^{1-k} w_k(q)\). 

\begin{lemma}\label{lemma2}\({}\)
\begin{itemize}
\item[(a)] For \(k \ge 1\) and \(q\in (0,1)\) the function \(w_k(q)\) is the unique solution of the fixed-point equation 
\begin{equation}\label{eqn_w_k_fixed_point}
w_k(q)=1+q F_k(w_k(q);q).
\end{equation}  
\item[(b)]
Fix \(k\ge 1\). Assume that for some \(\delta\in (0, 1/20]\) the function \(w_k^{(0)}(q)\) is analytic in \(|q|<\delta\) and satisfies \(|w_k^{(0)}(q) - 1|\le 1/2\) for all \(|q|\le \delta\). Define a sequence of functions \(\{w_k^{(m)}(q)\}_{m\ge 0}\) via
\begin{equation}\label{def_w^n}
w_k^{(m+1)}(q)=1+qF_k(w_k^{(m)}(q);q). 
\end{equation}
The functions \(w_k^{(m)}(q)\) are analytic in \(|q| <\delta\) and for every \(m\ge 1\) and \(|q| \le  \delta\) we have
\begin{equation}\label{w_contraction}
|w_k^{(m+1)}(q)-w_k^{(m)}(q)| \le 2^{-m}  |w_k^{(1)}(q)-w_k^{(0)}(q)|. 
\end{equation}
The functions \(w_k^{(m)}(q)\) converge to  \(w_k(q)\) as \(m \to +\infty\), uniformly on \(|q|\le \delta\). 
\end{itemize}
\end{lemma}
\begin{proof}
The proof of part (a) is based on ideas from \cite{Sokal_2009}.  Fix \(k\) and \(q\) and write \(x_k(q)=- k q^{1-k} w\). The equation \(f(x_k;q)=0\) is equivalent to
 \[
 f(x_k;q)=\sum\limits_{n\ge 0} (-1)^n\frac{ k^n}{n!} w^n 
 q^{(1-k)n+n(n-1)/2}=0,
 \]
which implies
 \[
 q^{k(k-1)/2} f(x_k;q)= \sum\limits_{n\ge 0} (-1)^n\frac{ k^n}{n!} w^n q^{(n-k+1)(n-k)/2}=0.
 \]
We separate the above series into two sums \(0\le n \le  k-2\) and \(n \ge k+1\) and isolate the terms with \(n=k-1\) and \(n=k\): 
\begin{align*}
(-1)^{k-1}  \frac{k^{k-1}}{(k-1)!} w^{k-1} + 
(-1)^{k} \frac{ k^{k}}{k!} w^{k}
&+  \sum\limits_{n=0}^{k-2} (-1)^n \frac{ k^n}{n!} w^n q^{(n-k)(n-k-1)/2}\\
&+ \sum\limits_{n\ge k+1} (-1)^n \frac{ k^n}{n!} w^n q^{(n-k)(n-k-1)/2}=0.
\end{align*}
Dividing both sides of this equation by \(k^{k-1} w^{k-1} /(k-1)!\) and rearranging the terms results in
\begin{align*}
w=1
& -  k! \sum\limits_{n=0}^{k-2} (-1)^{n-k} \frac{ k^{n-k}}{n!} w^{n-k+1} q^{(n-k+1)(n-k)/2}\\
&- k! \sum\limits_{n\ge k+1} (-1)^{n-k} \frac{ k^{n-k}}{n!} w^{n-k+1} q^{(n-k+1)(n-k)/2}.
\end{align*}
Now we change the index of summation \(n = k-1-i\) in the first sum and \(n = k+i\) in the second sum and we obtain
\begin{equation}\label{eq:g=1+2sums}
w=1
+  k! \sum\limits_{i=1}^{k-1} (-1)^{i} \frac{ k^{-i-1}}{(k-i-1)!} w^{-i} q^{i (i+1)/2}
- k! \sum\limits_{i\ge 1} (-1)^{i} \frac{k^{i}}{(k+i)!} w^{i+1} q^{i(i+1)/2}.
\end{equation}
Using the definition of \(\alpha_i(k)\) in \eqref{def_alpha_i}, we simplify 
\[
\frac{ k! k^{-i-1}}{(k-i-1)!}=\alpha_i(k), \;\;\; 
\frac{k! k^{i}}{(k+i)!}=\alpha_i(-k)^{-1}. 
\]
Noting that \(\alpha_i(k)=0\) for \(i\ge k\), we can finally
rewrite \eqref{eq:g=1+2sums} in the form
$w=1+q F_k(w;q)$. The uniqueness of the solution to this fixed-point equation follows from the Implicit Function Theorem. This ends the proof of part (a). 

To prove part (b), we first estimate $F_k(w;q)$ and its derivative $\partial_w F_k(w;q)$. We first check that for all \(i,k \in {\mathbb N}\) we have \(|\alpha_i(k)|<1\) 
and \(|\alpha_i(-k)^{-1}|<1\). Thus, for every \(k \in {\mathbb N}\), \(|w-1|<1/2\) and \(|q|<1/20\) we have 
\begin{equation}\label{estimate_Fk}
|F_k(w;q)| \le 
\sum\limits_{i\ge 1}  
\Big[   |w|^{-i} +  |w|^{i+1} \Big]
 |q|^{i (i-1)/2} < \sum\limits_{i\ge 1}  
\Big[   2^{i} +  (3/2)^{i+1} \Big]
 20^{-i (i-1)/2}=4.6204\cdots<5,
\end{equation}
and 
\begin{align}\label{estimate_Fkw}
|\partial_w F_k(w;q)| &\le 
\sum\limits_{i\ge 1}  
\Big[  i |w|^{-i-1} +  (i+1) |w|^{i} \Big]
 |q|^{i (i-1)/2} \\ \nonumber & < \sum\limits_{i\ge 1}  
\Big[   i 2^{i+1} + (i+1) (3/2)^{i} \Big]
 20^{-i (i-1)/2}= 8.1452\cdots<10.
 \end{align}
When deriving the above estimates we used the fact that \(|w-1|<1/2\) implies \(1/2<|w|<3/2\). 

Since \(|w_k^{(0)}(q)-1|<1/2\) for \(|q|<\delta\le 1/20\), the inequality \eqref{estimate_Fk}  implies that \(w_k^{(1)}(q)\) is analytic in \(|q|<\delta\). Moreover, for \(|q|<\delta\) we have \(F_k(w_k^{(0)}(q);q)|<5\), thus (with the help of \eqref{def_w^n}) we conclude that  \(|w_k^{(1)}(q)-1|<5|q|<5 \delta<1/2\).  In the same way, we can prove (by induction) that all functions \(w_k^{(m)}(q)\) are analytic in \(|q|<\delta\) and satisfy \(|w_k^{(m)}(q)-1|<1/2\).

Finally, to establish \eqref{w_contraction}, we note that for for \(m\ge 1\)  \(|q|\le \delta\)
\begin{align*}
|w_k^{(m+1)}(q)-w_k^{(m)}(q)|&=|q| \times |F(w_k^{(m)}(q);q)-F(w_k^{(m-1)}(q);q)| \\ & \le  \delta  \times 
\max\limits_{|w-1|<1/2} |\partial_w F_k(w;q)| \times |w_k^{(m)}(q)-w_k^{(m-1)}(q)|\\
& < \frac{1}{20} \times 10  \times |w_k^{(m)}(q)-w_k^{(m-1)}(q)|  =\frac{1}{2} |w_k^{(m)}(q)-w_k^{(m-1)}(q)|.
\end{align*}
This ends the proof of \eqref{w_contraction}. 
The contraction inequality \eqref{w_contraction} ensures that \(w_k^{(m)}(q)\) converges (uniformly on \(|q|\le \delta\) ) to $w_k(q)$ as \(m\to +\infty\). 
\end{proof}

\begin{lemma}\label{lemma3} There exists \(\delta \in (0,1)\) such that 
\begin{equation}\label{w_k_asymptotics}
w_k(q)=1+A_0(q) k^{-2} + O(k^{-3}), \;\;\; k \to +\infty,
\end{equation}
 uniformly in \(|q|\le \delta\). 
\end{lemma}
\begin{proof}
We recall that \(A_j(q)\) is defined via \eqref{def_Aj}. 
First, we establish the following result: as \(k \to \infty\)
\begin{equation}\label{alpha_asymptotics}
\alpha_i(k)=
1-\frac{i(i+1)}{2k}+\frac{(i-1) i (i+1) (3i+2)}{24 k^2} + 
O((i+1)! |k|^{-3}),
\end{equation}
uniformly in \(i \ge 1\) (in other words, the implied constant in the big-O term is absolute).  To prove this result, we recall that \(s(m,n)\) denote Stirling numbers of the first kind (see \cite{NIST}[Section 26.8] for their definition and various properties). We compute   
\begin{equation}\label{alpha_asymptotics_proof}
\alpha_i(k)=k^{-i-1} (k)_{i+1}= 
\sum\limits_{l=0}^{i+1} s(i+1,l) k^{l-i-1}=
1-\frac{i(i+1)}{2k}+\frac{(i-1) i (i+1) (3i+2)}{24 k^2} + 
{\mathcal E}_i(k),
\end{equation}
where 
\[
|{\mathcal E}_i(k)| = \sum\limits_{l=0}^{i-2}
s(i+1,l) k^{l-i-1} < (i+1)! |k|^{-3}.
\]
The first equality in \eqref{alpha_asymptotics_proof} follows from the definition of Stirling numbers of the first kind (formula (26.8.7) in \cite{NIST}). The second equality \eqref{alpha_asymptotics_proof}
requires explicit expressions for \(s(i+1,l)\) for \(l \in \{i-1,i,i+1\}\), which can be found in the subsection 26.8(iii) in \cite{NIST}. The upper bound for \({\mathcal E}_i(k)\) follows from 
the fact that 
\[
\sum\limits_{l=0}^{i+1} |s(i+1,l)|=(i+1)!, 
\]
see formula (26.8.29	) in \cite{NIST}. 

From \eqref{alpha_asymptotics}, after some simplification, we derive the following asymptotic result: 
\begin{equation}\label{product_alpha_asymptotics}
\alpha_i(k) \alpha_i(-k)=1-\frac{i(i+1)(2i+1)}{6k^2} + 
O((i+1)!^2 k^{-3}), \;\;\; k\to +\infty,
\end{equation}
which also holds uniformly in \(i \ge 1\).

Since \(A_0(q)\) is analytic in \(|q|<1\) and satisfies \(A_0(0)=0\) (see \eqref{def_Aj}), we can find \(\delta\le 1/20\) small enough, such that \(|A_0(q)|\le 1/2\) for all \(|q| \le \delta\). We define \(w_k^{(0)}(q)=1-A_0(q)k^{-2}\) and check that \(|w_k^{(0)}(q)-1|\le 1/2\) for all \(k\ge 1\) and \(|q|\le \delta\). We  define the sequence of functions \(w_k^{(m)}(q)\) via   \eqref{def_w^n}. 

Our next goal is to obtain an upper bound for \(|w_k^{(1)}(q)-w_k^{(0)}(q)|\). Recalling formulas \eqref{def_w^n} and \eqref{def_F_k_wq}, we compute 
\begin{align}
\label{bound_proof1}
&k^3 (w_k^{(1)}(q)-w_k^{(0)}(q))=k^3 q F_k(1+A_0(q)k^{-2};q)-kA_0(q) \\
\nonumber
&=
\sum\limits_{i\ge 1} (-1)^i 
\frac{\alpha_i(-k)^{-1}}{ (1+A_0(q)k^{-2})^{i}} k^3 \Big[ \alpha_i(k) \alpha_i(-k)  - (1+A_0(q)k^{-2})^{2i+1}\Big]q^{i (i+1)/2}
-k A_0(q)
\end{align}
Using the Binomial Theorem, the fact that 
\[
\sum\limits_{l=0}^{2i+1} \binom{2i+1}{l}=2^{2i+1}
\]
and the upper bound \(|A_0(q)|<1/2\) for \(|q|\le \delta\), we check that 
\begin{align*}
(1+A_0(q)k^{-2})^{2i+1}&=1+(2i+1) A_0(q)k^{-2} + \sum\limits_{l=2}^{2i+1} \binom{2i+1}{l} A_0(q)^l k^{-2l}\\
&=
1+(2i+1) A_0(q)k^{-2}+
O(4^{i} k^{-4}), \;\;\; k\to +\infty, 
\end{align*}
uniformly in \(i\ge 1\) and \(|q|\le \delta\). Similarly, 
\[
(1+A_0(q)k^{-2})^{-i}=1+O(2^i k^{-2}), \;\;\; k\to +\infty,
\]
uniformly in \(i\ge 1\) and \(|q|\le \delta\).
Combining the above results with 
\eqref{alpha_asymptotics}, \eqref{product_alpha_asymptotics} and 
\eqref{bound_proof1} we obtain
\begin{align}
\nonumber
k^3 (w_k^{(1)}(q)-w_k^{(0)}(q))&=
k 
\sum\limits_{i\ge 1} (-1)^i 
  \Big[ -\frac{1}{6} i(i+1)(2i+1)  - (2i+1)A_0(q) \Big]q^{i (i+1)/2}
\\
\label{k^3_bound1}
&-k A_0(q)+ O(1),
\end{align}
as \(k\to +\infty\), uniformly in \(|q|\le \delta\).
 
Now comes the crucial step in the proof: the function \(A_0(q)\) satisfies an identity
\begin{equation}\label{k^3_bound2}
A_0(q)=\sum\limits_{i\ge 1} (-1)^{i+1} \Big[ 
 \frac{1}{6} i(i+1)(2i+1) + (2i+1) A_0(q)\Big] q^{i(i+1)/2}, \;\;\; |q|<1. 
\end{equation}
The above result is derived  by taking the logarithmic derivative of Jacobi's triple product identity
\[
\prod\limits_{i\ge 1} (1-q^i)^3=
\sum\limits_{i\ge 1} (-1)^{i+1} (2i-1) q^{i(i-1)/2},
\]
see the proof of Lemma 2 and Theorem 1 in \cite{Zhang_2016}. Combining \eqref{k^3_bound1} and \eqref{k^3_bound2} we obtain the following result: There exists some absolute constant \(C>0\) such that \(k^3 |w_k^{(1)}(q)-w_k^{(0)}(q)|<C\) for all \(k\ge 1\) and \(|q|\le \delta\). Using \eqref{w_contraction} and the telescoping sum argument, this implies an upper bound
\[
k^3 |w_k^{(m+1)}(q)-w_k^{(0)}(q)| \le 
2 k^3  |w_k^{(1)}(q)-w_k^{(0)}(q)| \le 2 C 
\] 
for all \(m\ge 0\), \(k\ge 1\) and \(|q|\le \delta\). Passing to the limit as \(m\to +\infty\) gives us the desired result.  
\end{proof}

\noindent
{\bf Proof of Theorem \ref{thm_main}:}
We recall that \({\mathbb Z}[x]\) denotes the set of polynomials in \(x\)-variable with integer coefficients and that the polynomials \(\{Q_n\}_{n\ge 1}\) are defined via
\eqref{def_gamma_nk_si}. 
We will require the following properties of these polynomials: 
\begin{itemize}
\item[(a)] \(Q_{n_2}(x)/Q_{n_1}(x) \in {\mathbb Z}[x]\) for all \(1\le n_1 < n_2\),
\item[(b)] \(Q_n(x)/(Q_j(x) Q_{n-j}(x)) \in {\mathbb Z}[x]\) for all \(1 \le j \le n-1\). 
\end{itemize}
The first property follows from the definition \eqref{def_gamma_nk_si} and  the fact that \(\gamma_{n_1,l}\le \gamma_{n_2,l}\) for \(n_1<n_2\). The second property follows from the inequality \(\gamma_{j,l} + \gamma_{n-j,l} \le \gamma_{n,l}\).

We define the class \({\mathfrak A}\) of formal power series 
\[
f(q,k)=1+\sum\limits_{n\ge 1} R_n(k) q^n
\]
by requiring  \( R_n(k)Q_n(k) \in {\mathbb Z}[k]\) for all \(n\ge 1\).  In other words, \(f\in {\mathfrak A}\) if the coefficient in front of \(q^n\) (denoted by \(R_n(k)=[q^n] f\)) is a rational function of \(k\) with denominator \(Q_n(k)\) and  numerator in \({\mathbb Z}[k]\). 
We verify that this class \({\mathfrak A}\) is closed with respect to multiplication and division: 
\begin{itemize}
\item[(c)] If \(f_i(q,k) \in {\mathfrak A}\) for \(i=1,2\), then  \(f_1(q,k) f_2(q,k) \in {\mathfrak A}\). 
\item[(d)] If \(f(q,k) \in {\mathfrak A}\), then \(1/f(q,k)\in {\mathfrak A}\). 
\end{itemize}
To prove (c), let \(R^{(i)}_n(k)\) be the coefficient in front of \(q^n\) of the power series \(f_i(q,k)\). 
Then the coefficient of \(q^n\) of \(f_1(q,k) f_2(q,k)\) is 
\[
R_n(k)=R^{(1)}_n(k)+R^{(2)}_n(k)+
\sum\limits_{j=1}^{n-1} R^{(1)}_j(k)R^{(2)}_{n-j}(k).
\]
We need to check that \(R_n(k) Q_n(k) \in {\mathbb Z}[k]\).
Note that \(R^{(i)}_n(k)Q_n(k) \in {\mathbb Z}[k]\), by our assumption \(f_i(q;k) \in {\mathfrak A}\). We also have 
\[
Q_n(k)R^{(1)}_j(k)R^{(2)}_{n-j}(k)=\frac{Q_n(k)}
{Q_{j}(k) Q_{n-j}(k)} \times R^{(1)}_j(k)Q_j(k) \times R^{(2)}_{n-j}(k) Q_{n-j}(k) \in {\mathbb Z}[k],
\]
due to property (b) above. The proof of item (d) uses the same ideas and is left to the reader. 

Now we set \(w_k^{(0)}(q)=1\) and define the sequence of functions \(w_k^{(m)}(q)\) via \eqref{def_w^n}.  Our goal is to prove that \(w_k^{(m)}(q) \in {\mathfrak A}\). We will do this by induction. The claim is clearly true for \(w_k^{(0)}(q)=1\). Assume that \(w_k^{(m)}(q) \in {\mathfrak A}\) for some \(m\ge 0\). Then, from the above facts (c) and (d), we conclude that \((w_k^{(m)}(q))^{-i}\in {\mathfrak A}\) and \((w_k^{(m)}(q))^{i+1} \in {\mathfrak A}\) for all \(i\ge 1\). Using the fact \(\alpha_i(k) k^i \in {\mathbb Z}[k]\), we check that 
\begin{equation}\label{def_h_i1}
h_i(q;k):=1 + \alpha_i(k) (w_k^{(m)}(q))^{-i} q^{i(i+1)/2}, \;\;\; i\ge 1,
\end{equation}
 also belongs to \({\mathfrak A}\). Next, we will show that  
\begin{equation}\label{def_h_i2}
H_i(q;k)=1+\alpha_i(-k)^{-1} 
(w_k^{(m)}(q))^{i+1} q^{i(i+1)/2}, \;\;\; i\ge 1, 
\end{equation}
 is in \({\mathfrak A}\). Note that 
\[
\alpha_i(-k)^{-1}=\frac{k^i}{(k+1)(k+2)\dots (k+i)}.
\]
Since \((w_k^{(m)}(q))^{i+1} \in {\mathfrak A}\), we can write it in the form
\[
(w_k^{(m)}(q))^{i+1}=1+\sum\limits_{n\ge 1} \frac{\widetilde P_n(k)}{Q_n(k)} q^n,
\]
for some polynomials \(\widetilde P_
n(k) \in {\mathbb Z}[k]\).
Thus, when \(1\le n<i(i+1)/2\),  the coefficient in front of \(q^n\) of the series expansion of \(H_i(q;k)\) will be zero. The coefficient in front of \(q^{i(i+1)/2}\) of the series expansion of \(H_i(q;k)\) will be \(\alpha_i(-k)^{-1}\), and it is easy to check that \(Q_{i(i+1)/2}(k)
\alpha_i(-k)^{-1} \in {\mathbb Z}[k]\), since the term \((k+l)\) (with \(1\le l \le i\)) will appear in this expression raised to the power
\[
\gamma_{n,l}-1=\Big\lfloor \frac{i(i+1)}{l(l+1)} \Big\rfloor-1 \ge 0. 
\] 
For \(n>i(i+1)/2\), the coefficient in front of \(q^n\) of the series expansion of \(H_i(q;k)\) will be
\[
\frac{k^i}{(k+1)(k+2)\dots (k+i)} \times 
\frac{\widetilde P_{n-i(i+1)/2}(k)}{Q_{n-i(i+1)/2}(k)},
\]
and we need to verify that 
\begin{equation}\label{eq_Qn_Z}
\frac{Q_n(k)}{(k+1)(k+2)\dots (k+i)Q_{n-i(i+1)/2}(k)} \in {\mathbb Z}[k]. 
\end{equation}
The above statement can be established by considering 
the term \((k+l)\) (with \(1\le l \le i\)) in the above expression.
This term appears in the numerator raised to power \(\gamma_{n,l}\), and it appears in the denominator raised to the power  
\[
1+\gamma_{n-i(i+1)/2,l}=1+\Big\lfloor \frac{2n}{l(l+1)}-\frac{i(i+1)}{l(l+1)} \Big\rfloor
\le \Big\lfloor \frac{2n}{l(l+1)} \Big\rfloor=\gamma_{n,l}.
\]
Thus \eqref{eq_Qn_Z} holds true and we have proved that \(H_i(q;k) \in {\mathfrak A}\). 

Now we can complete the induction step. We have assumed that \(w_k^{(m)}(q) \in {\mathfrak A}\) and demonstrated that \(h_i(q;k)\) and \(H_i(q;k)\), as defined in \eqref{def_h_i1} and \eqref{def_h_i2}, also belong to \({\mathfrak A}\). 
Thus, for any \(I \ge 1\), we have 
\[
1+\sum\limits_{i=1}^{I} (-1)^i 
\big[ h_i(q;k) - H_i(q;k)] \in {\mathfrak A}. 
\]
Note that adding one more term \((-1)^{I+1}(h_{I+1}(q;k) - H_{I+1}(q;k))\) to the above sum will not affect the coefficient in front of \(q^n\) with \(n \le  I(I+1)/2\) (see 
\eqref{def_h_i1} and \eqref{def_h_i2}). Since 
\[
w_k^{(m+1)}(q)=1+\sum\limits_{i\ge 1} (-1)^i 
\big[ h_i(q;k) - H_i(q;k)], 
\]
which follows from \eqref{def_F_k_wq} and \eqref{def_w^n}, we conclude that \(w_k^{(m+1)}(q) \in {\mathfrak A}\). This completes the induction step. 

Thus, we have proved that for every \(m\ge 0\) 
\[
w_k^{(m)}(q)=1+\sum\limits_{n\ge 0} \frac{P_n^{(m)}(k)}{Q_n(k)} q^n,
\]
for some polynomials \(P_n^{(m)}(k) \in {\mathbb Z}[k]\). 
From \eqref{def_w^n} it is straightforward to establish that  \(w_k^{(m)}(q)-w_k(q)=O(q^{m+1})\) as \(q\to 0\). Therefore, \(w_k^{(m)}(q)\) and \(w_k(q)\) have identical first \(m+1\) terms of Taylor series at zero. In particular, the polynomials \(P_n^{(m)}(k)\) are the same for all \(m \ge n\). This proves that \(w_k(q)\) also belongs to \({\mathfrak A}\), so that there exist polynomials \(\{P_n\}_{n\ge 1}\) with integer coefficients, such that 
\[
w_k(q)=1+\sum\limits_{n\ge 1} \frac{P_n(k)}{Q_n(k)} q^n. 
\]  
This, and the fact that \(1/w_k(q) \in {\mathfrak A}\), ends the proof of part (i).

Next, we will prove the results concerning the degree of \(Q_n\). We have 
\[
n^{-1} M_n = 1 + \sum\limits_{l\ge 1} \frac{\gamma_{n,l}}{n} < 
1 + 2 \sum\limits_{l\ge 1} \frac{1}{l(l+1)}=
1+2\sum\limits_{l\ge 1} \Big[ \frac{1}{l} - \frac{1}{l+1}\Big] = 3. 
\]
On the other hand, for every \(m\ge 1\)
\begin{align*}
\liminf\limits_{n \to +\infty} 
n^{-1} M_n
&= 1+ \liminf\limits_{n \to +\infty} 
\frac{1}{n} \sum\limits_{l\ge 1}  \gamma_{n,l}
\ge 1+ \liminf\limits_{n \to +\infty} 
\frac{1}{n} \sum\limits_{l=1}^m  \gamma_{n,l}
\\
&= 1+\sum\limits_{l=1}^m \frac{2}{l(l+1)}
= 1 + 2\sum\limits_{l=1}^m \Big[ \frac{1}{l} - \frac{1}{l+1}\Big]= 3-\frac{2}{m+1}.
\end{align*}
The above result and the upper bound 
\(n^{-1} M_n < 3\) imply the result 
\(n^{-1} M_n \to 3\) as \(n\to +\infty\).

Let us now prove \eqref{form_Pn}. 
The asymptotic formula \eqref{Wang_Zhang_result} states that 
\begin{equation}\label{k^3_formula}
k^3 (w_k(q)-1-A_0(q) k^{-2}) \to -A_1(q), \;\;\; k\to +\infty, 
\end{equation} 
 for all \(q\in (0,1)\). Lemma \ref{lemma3} ensures that the functions \(k^3 (w_k(q)-1-A_0(q)k^{-2})\) are uniformly bounded for \(k \ge 1\) and \(|q|\le \delta\).  By the Vitali-Porter Theorem, the convergence in \eqref{k^3_formula} holds uniformly for  \(|q|\le \delta\), which implies the convergence of the corresponding Taylor coefficients. Thus, for \(n\ge 1\), we have  
\[
\lim_{k\to +\infty} k^3 \Big[\frac{P_n(k)}{Q_n(k)} - \sigma(n)k^{-2} \Big] = -n \sigma(n), 
\]
or, equivalently, 
\begin{equation}\label{P_nQ_n_asymptotics}
\frac{P_n(k)}{Q_n(k)}=\sigma(n)k^{-2} - n \sigma(n) k^{-3} + o(k^{-3}), \;\;\; k\to +\infty. 
\end{equation}
The above result proves that \({\textnormal{deg}}(P_n)=
{\textnormal{deg}}(Q_n)-2\) and that the leading coefficient of \(P_n\) is \(\sigma(n)\). 

To find the second leading coefficient of \(P_n\), we use the expansion
\begin{equation}\label{Q_leading_terms}
Q_n(k)=k^{M_n}+ \mu_1(n)  k^{M_n-1} +
\frac{1}{2} \big( \mu_1(n)^2-\mu_2(n)\big) 
k^{M_n-2} + {\textnormal{ lower order terms }},
\end{equation}
which follows from the definition \eqref{def_gamma_nk_si}. 
Combining this expansion with the asymptotic result \eqref{P_nQ_n_asymptotics}, we obtain the second coefficient in \eqref{form_Pn}. 

To verify that the second coefficient in \eqref{form_Pn} is positive for \(n\ge 3\), we compute it for \(n=3\) and \(n=4\), obtaining the values of 8 and 14. For \(n\ge 5\), we estimate 
\[
-n + \sum\limits_{l\ge 1} l \gamma_{n,l} \ge 
-n + \sum\limits_{l=1}^{2} l \gamma_{n,l} \ge 
-n + \sum\limits_{l=1}^{2} l \Big[ \frac{2n}{l(l+1)} - 1 \Big]
=\frac{2n}{3} - 3>0.
\]
Thus the second coefficient in  \eqref{form_Pn} is positive for \(n\ge 3\).
\qed

\label{page_sketch_proof}

We now outline the proof of \eqref{coeff3_Pn}. First, we establish the stronger version of \eqref{w_k_asymptotics} in the form
\[
w_k(q)=1+B_1(q) k^{-2} + B_2(q) k^{-3} + O(k^{-4}), \;\;\; k\to +\infty,
\]
uniformly in some neighbourhood of \(q=0\) (the coefficients \(B_i(q)\) are defined in \eqref{formulas_Ci}). This can be done by applying the same method as in the proof of Lemma \ref{lemma3}, but starting with the initial approximation \[
w_k^{(0)}(q)=1+
B_1(q) k^{-2} + B_2(q) k^{-3}
\] and showing that \(k^4|w_k^{(1)}(q)-w_k^{(0)}(q)|\) is uniformly bounded for \(k\ge 1\) and  \(q\) in some neighbourhood of zero. Using the same argument as in the proof of Theorem \ref{thm_main}(ii), we conclude that
\[
\frac{P_n(k)}{Q_n(k)}=\sigma(n) k^{-2} - n \sigma(n) k^{-3}
+ [q^n]B_3(q) k^{-4}+O(k^{-5}), \;\;\; k\to +\infty,
\]
where \([q^n]B_3(q)\) is the coefficient in front of \(q^n\) of 
the power series \(B_3(q)\), which is given in \eqref{formulas_Ci}. 
Formulas \eqref{form_Pn} and \eqref{Q_leading_terms} imply that
\begin{align*}
\frac{P_n(k)}{Q_n(k)}&=k^{-2} \frac{\sigma(n)+\sigma(n)\big(\mu_1(n)-n\big) k^{-1} + \xi k^{-2} + O(k^{-3})}
{ 1+ \mu_1(n)  k^{-1} +
\frac{1}{2} \big( \mu_1(n)^2-\mu_2(n)\big) 
k^{-2}+O(k^{-3}) }\\&=\sigma(n) k^{-2} - n \sigma(n) k^{-3}
+ [q^n]B_3(q) k^{-4}+O(k^{-5}),
\end{align*}
and from here one can obtain the coefficient \(\xi\) as given in \eqref{coeff3_Pn}. To prove the corresponding result for \(\widehat P_n\), we would use the same techniques applied to the asymptotic expansion 
\[
1/w_k(q)=1-B_1(q) k^{-2} - B_2(q) k^{-3} - \big( B_3(q)-B_1(q)^2\big) k^{-4} + O(k^{-5}). 
\]

%**************************************************************************
%**************************************************************************
%**************************************************************************

 \section{Investigating the polynomials \(P_n\) and \(\hat P_n\)}\label{section_numerics}

 %**************************************************************************
 %**************************************************************************
 %**************************************************************************

\label{our_algorithm}

Let us describe our algorithm for computing the polynomials \(P_n\) and \(\hat P_n\). The main idea is to initialize \(w_k^{(0)}(q)=1\) and iteratively compute \(w_k^{(m)}(q)\) for \(m\ge 1\) via \eqref{def_w^n}. 
Another important ingredient is the following recurrence relation 
\begin{equation}
\label{recursion_hat_P_n}
\widehat P_n(k)=P_n(k) - Q_n(k)\sum\limits_{j=1}^{n-1} 
  \frac{ \widehat P_j(k) P_{n-j}(k)}{Q_j(k)Q_{n-j}(k)}, n\ge 1, 
\end{equation}
which is derived by comparing the coefficients in front of \(q^n\) in the identity \(w_k(q) \times (1/w_k(q))=1\).  
Now, applying \eqref{def_w^n} with \(w_k^{(0)}(q)=1\) we compute 
\begin{equation*}
w_k^{(1)}(q)=1+\frac{1}{k(k+1)}q+O(q^2),  \;\;\; q\to 0,
\end{equation*}
which implies \(P_1(k)=1\). Formula \eqref{recursion_hat_P_n} then gives us \(\widehat P_1(k)=1\). Applying iteration \eqref{def_w^n} once again, we compute 
\begin{equation*}
w_k^{(2)}(q)=1+\frac{1}{k(k+1)}q+\frac{3k^2-1}{k^2(k+1)^2}q^2+O(q^3),
\end{equation*} 
yielding \(P_2(k)=3k^2-1\), and from \eqref{recursion_hat_P_n} we obtain \(\widehat P_2(k)=3k^2-2\). In the next iteration, we have
\begin{equation*}
w_k^{(3)}(q)=1+\frac{1}{k(k+1)}q+\frac{3k^2-1}{k^2(k+1)^2}q^2+\frac{4k^5+8k^4-7k^3-4k^2+6k+4}{k^3(k+1)^3(k+2)}q^3+ O(q^4),
\end{equation*} 
which gives \(P_3(k)\) and then \(\widehat P_3(k)\) via \eqref{recursion_hat_P_n}. In general, \(m\)-th iteration of \eqref{def_w^n} yields the polynomial \(P_m\), which is then used to find \(\widehat P_m\) via 
\eqref{recursion_hat_P_n}.    
The first ten polynomials \(P_n\) and \(\widehat P_n\) can be easily computed using any symbolic computation software. We present them in Appendix \ref{AppendixB}. 

\begin{remark}\label{remark_shifted}{
\normalfont
Our polynomials \(P_n\) and \(\widehat P_n\)  differ from those computed by Sokal \cite{Sokal_2009}. This discrepancy arises because we start indexing the zeros \(x_k\) of the deformed exponential function at   \(k=1\), whereas Sokal \cite{Sokal_2009} starts at \(k=0\). Thus, the polynomials in \cite{Sokal_2009} correspond to \(P_n(k+1)\) and \(\widehat P_n(k+1)\). For example, 
\[
P_2(k+1)=3k^2 + 6k + 2, \;\;\;
P_3(k+1)=4k^5 + 28k^4 + 65k^3 + 63k^2 + 29k + 11, \;\;\;
{\textnormal{ etc.}}
\]}
\end{remark}

\begin{remark}\label{remark2}{
\normalfont
An alternative recursive method for computing polynomials \(P_n\) and \(\widehat P_n\) can be formulated, which may offer advantages for theoretical analysis.  We denote  \({\mathbf n}=\{n_j\}_{j\ge 0}\), where \(n_j\) are non-negative integers,
and for \(l, m \in {\mathbb N}\) we define
\begin{equation*}
{\mathcal A}(l,m):=\Big\{ {\mathbf n} : \sum\limits_{j\ge 0} n_j=l, \; 
\sum\limits_{j\ge 1} j n_j=m \Big\}.  
\end{equation*}
The set \({\mathcal A}(l,m)\) is finite and \({\mathbf n} \in  {\mathcal A}(l,m)\) implies \(n_j=0\) for \(j\ge m+1\). Applying the Multinomial Theorem and extracting the coefficient of \(q^n\) on both sides of 
\eqref{eqn_w_k_fixed_point}, we obtain 
\begin{align}
\label{recursion_P_n}
P_n(k)
&= Q_n(k) \sum_{\substack{i\ge 1 \\ i(i+1)/2<n}}
 \bigg[ i! \sum\limits_{{\mathbf n} \in {\mathcal A}(i,n-i(i+1)/2)}
\alpha_i(k) 
\frac{(-1)^{n_0}}{n_0!} \prod\limits_{j\ge 1} 
\frac{\widehat P_j(k)^{n_j}}{n_j! Q_j(k)^{n_j}}\\
\nonumber
&+ 
(-1)^{i+1} (i+1)! \sum\limits_{{\mathbf n} \in {\mathcal A}(i+1,n-i(i+1)/2)}
\alpha_i(-k)^{-1} 
\frac{1}{n_0!} \prod\limits_{j\ge 1} 
\frac{P_j(k)^{n_j}}{n_j! Q_j(k)^{n_j}}\bigg].
\end{align}
Formulas \eqref{recursion_hat_P_n}
and \eqref{recursion_P_n} give a recursive procedure for calculating all polynomials \(P_n\) and \(\widehat P_n\) starting from \(P_1=\widehat P_1=1\). Indeed, assume that \(n\ge 2\) and we have computed all polynomials \(\{P_j, \widehat P_j\}_{1\le j \le n-1}\). We note that the indices \({\mathbf n}=\{n_j\}_{j\ge 0}\) appearing in the sets 
\({\mathcal A}(i,n-i(i+1)/2)\) and \( {\mathcal A}(i+1,n-i(i+1)/2)\) have \(n_j=0\) for \(j\ge n\), thus the right-hand side of \eqref{recursion_P_n} depends only on polynomials \(\{P_j, \widehat P_j\}_{1\le j \le n-1}\). This allows us to compute \(P_n\) via    \eqref{recursion_P_n} and then we can obtain \(\widehat P_n\) via \eqref{recursion_hat_P_n}. 
}
\end{remark}

\begin{figure}[t]
\centering
{\includegraphics[height =8.5cm]{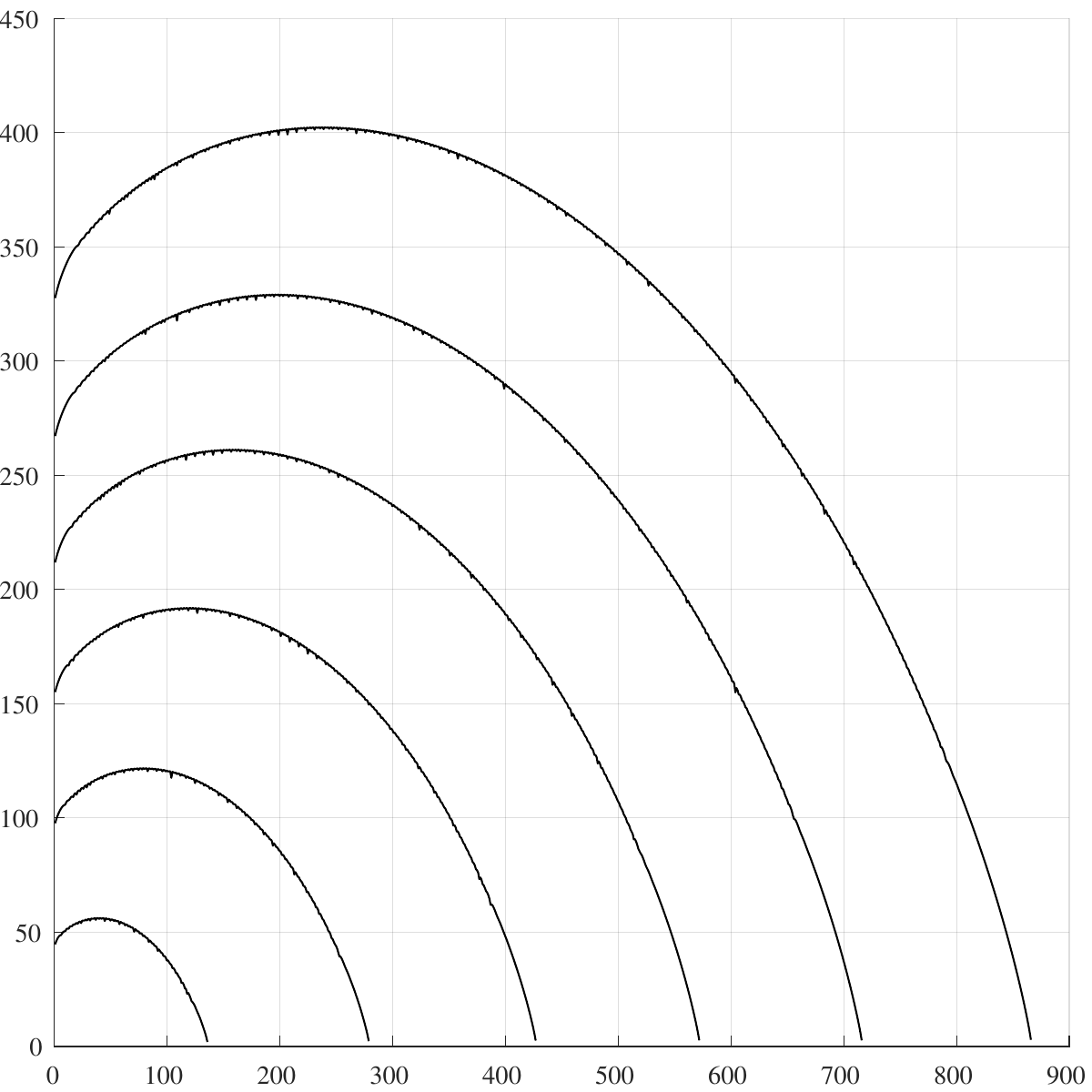}} 
\caption{The number of decimal digits of the coefficients of \(P_n\). For each value of \(n \in \{50,100,150,200,250,300\}\),  we plot \(\log_{10}|b_{n,i}|\) (on the \(y\)-axis) vs the index \(i=0,1,2,\dots,{\textnormal{deg}}(P_n)\) on the \(x\)-axis. }
\label{fig1}
\end{figure}

Using the algorithm based on \eqref{def_w^n}, that we described on page \pageref{our_algorithm}, we computed the coefficients of polynomials \(P_n\) and \(\widehat P_n\) for all \(1\le n\le 300\). These coefficients are available for download on  the author's webpage 
\url{https://kuznetsov.mathstats.yorku.ca/code/}. This calculation required high precision due to the rapid growth of the coefficients of \(P_n\) and \(\widehat P_n\). We used D. H. Bailey's arbitrary precision Fortran-90 package MPFUN2020 \cite{Bailey_2020}. The computation took approximately three weeks, running on a single core of a regular desktop computer having an Intel Core i5-10400 processor and 8GB of RAM.

 Figure \ref{fig1} provides insights into the magnitude  of the coefficients of polynomials \(P_n\) and \(\widehat P_n\). We write 
\[
P_n(k)=\sum\limits_{i=0}^{M_n-2} b_{n,i} k^i,
\]
where \(M_n\) is defined in \eqref{def_M_n}. The corresponding coefficients for \(\widehat P_n\) are denoted by \(\hat b_{n,i}\). 
The six graphs in Figure \ref{fig1} display  \(\log_{10}|b_{n,i}|\) for  \(n\in \{50,100,150,200,250,300\}\) and \(i=0,1,2,\dots,{\textnormal{deg}}(P_n)=M_n-2\), with \(i\) plotted along the \(x\)-axis. These values approximate the number of decimal digits required to represent  \(b_{n,i}\). We see from these graphs that the largest coefficients of \(P_{300}\) are on the order of \(10^{405}\).  The corresponding graphs showing the size of coefficients of \(\widehat P_n\) are very similar and we do not include them here.  Figure \ref{fig1} also demonstrates a rather regular distribution of the absolute value of the coefficients \(b_{n,i}\), across both 
 \(i\) and \(n\). This contrasts sharply 
 with the distribution of the signs of \(b_{n,i}\), which we discuss next. 

\begin{figure}[t!]
\centering
{\includegraphics[height =9.1cm]{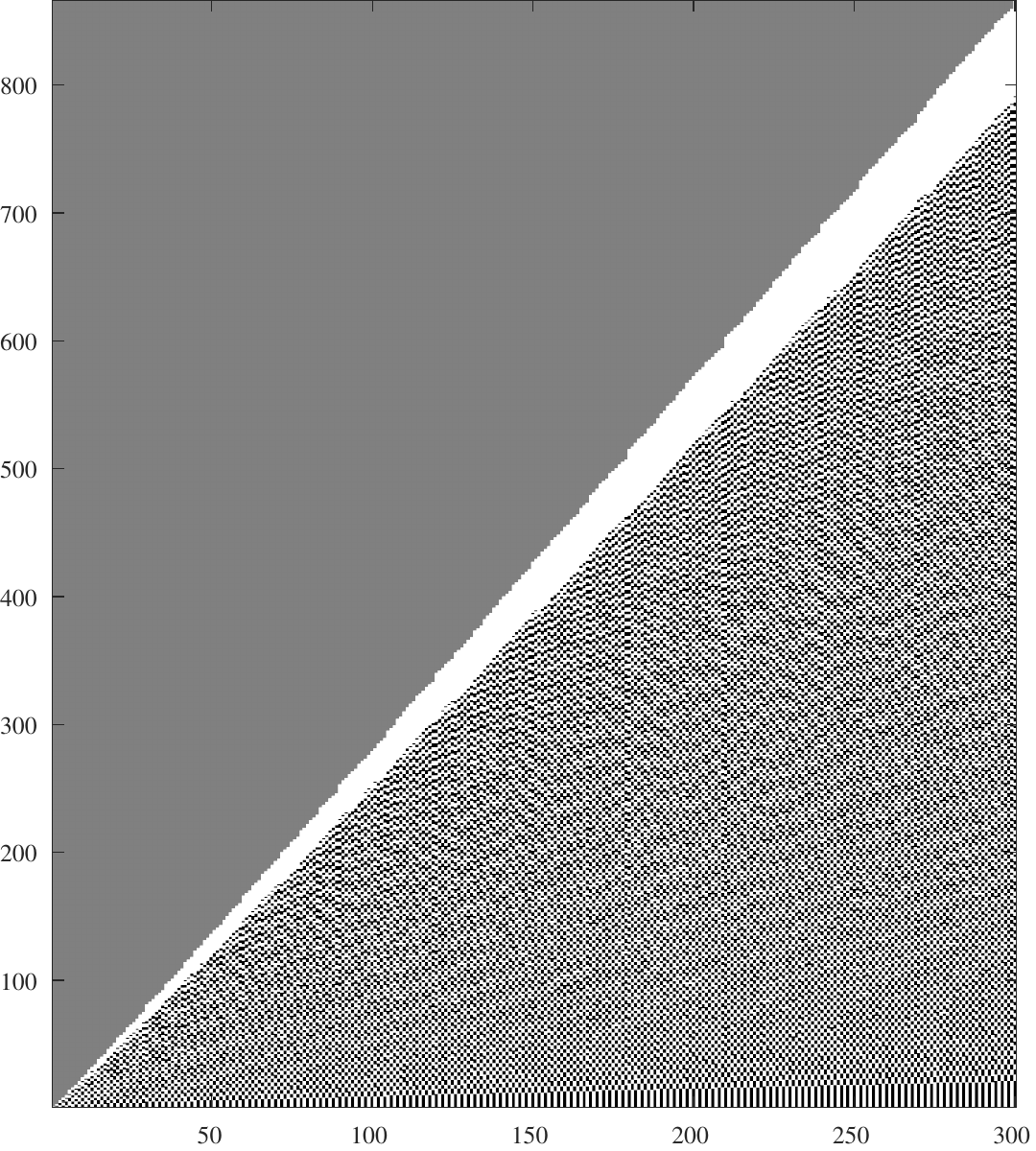}} 
{\includegraphics[height =9.1cm]{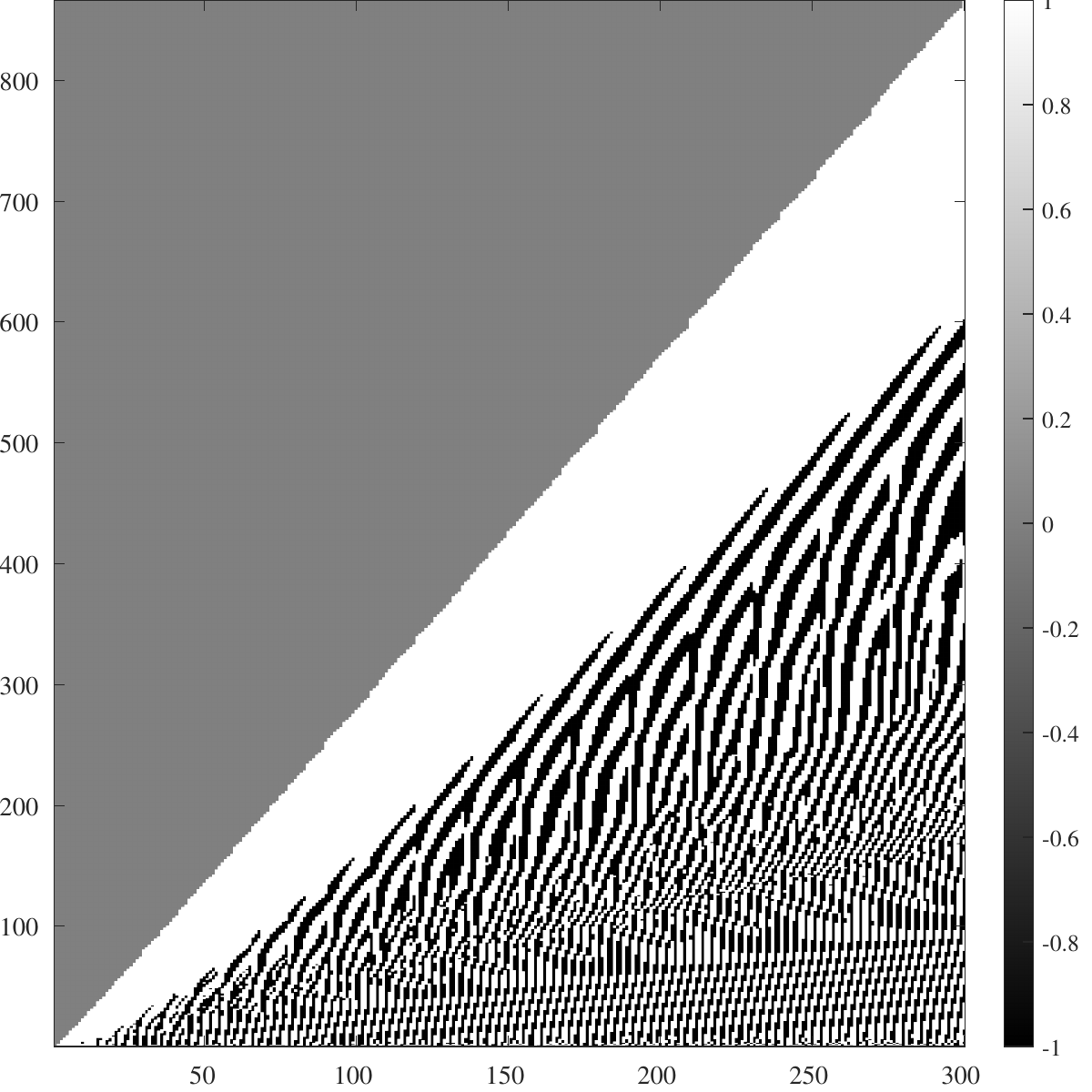}}
\caption{The sign of coefficients \(b_{n,i}\) (left) and \(B_{n,i}\) (right). Here \(n=1,2,\dots,300\) is on the \(x\)-axis and \(i=0,1,\dots,M_n-2={\textnormal{deg}}(P_n)\) is on the \(y\)-axis. } 
\label{fig2}
\end{figure}  

Figure \ref{fig2} presents the sign distribution of
 \(b_{n,i}\) (left) and \(B_{n,i}\) (right), where
\[
P_n(k+1)=\sum\limits_{i=0}^{M_n-2} B_{n,i} k^i. 
\]
Recall that this would be the form of polynomials if we started indexing zeros at \(k=0\) instead of \(k=1\), as discussed in
Remark \ref{remark_shifted}.  
The corresponding coefficients for \(\widehat P_n(k+1)\) will be denoted by \(\widehat B_{n,i}\).
In Figure \ref{fig2},  \(n=1,2,\dots,300\) is plotted along the \(x\)-axis and the \(i=0,1,\dots,M_n-2\) is along the \(y\)-axis. On the left graph, at each coordinate \((n,i)\) we plot a white pixel if \(b_{n,i}>0\) and a black pixel if  \(b_{n,i}>0\)
(if \(b_{n,i}=0\) we plot a gray pixel). The boundary of the gray region in the upper left of each plot represents the graph of \(M_n-2={\textnormal{deg}}(P_n)\). The right graph similarly shows the sign distribution of \(B_{n,i}\). 

We observe two interesting phenomena from Figure \ref{fig2}. First of all, both polynomials \(P_n(k)\) and \(P_n(k+1)\) have a large number of positive leading coefficients (that is, \(b_{n,i}>0\) and \(B_{n,i}>0\) for \(i\) close to \(\deg(P_n)=M_n-2\)).  Second, the signs of coefficients \(b_{n,i}\) and \(B_{n,i}\) have very different behavior. The sign of \(b_{n,i}\) follows a checkerboard pattern, with alternating cells of white/black color (though the cells are not all of the same size and the pattern is somewhat irregular). At the same time, the sign of \(B_{n,i}\) appears to follow a striped or ``zebra" pattern.

The corresponding graphs for the polynomials \(\widehat P_n(k)\) and \(\widehat P_n(k+1)\) are shown on Figure \ref{fig3}. The qualitative behavior is consistent with Figure \ref{fig2}, but two distinctions are apparent. First of all, the white region of positive leading coefficients is narrower, indicating that fewer leading coefficients \(\hat b_{n,i}\) and \(\widehat  B_{n,i}\) are positive. Second, we see sporadic black dots (negative coefficients) in the white region of leading coefficients. These will have implications for confirming numerically Sokal's Conjecture 3 (that is, checking that \(\widehat P_n(k)\ge 0\) for all \(k\ge 1\)), to which we now turn our attention.

 Since the leading coefficients of \(P_n\) and \(\widehat P_n\) are positive, it is clear that for every \(n\) we have \(P_n(k)\ge 0\) and \(\widehat P_n(k)\ge 0\)  for all \(k\) large enough. The following theorem provides an explicit upper bound for the largest real root of a polynomial.

\begin{theorem}[\bf Lagrange]
Let \(P(x)=\sum\limits_{j=0}^n a_j x^j\) be a real polynomial of degree \(n\) with \(a_n>0\). Define \(\nu(P)\) as the sum of the largest and the second largest numbers in the set
\[
\Big\{ \sqrt[i]{|a_{n-i}|/a_n} :
1\le i \le n \; {\textnormal{ and }} \;
a_{n-i}<0
\Big\}. 
\]
Then \(P(x)> 0\) for all
\(x > \nu(P)\). 
\end{theorem}

\begin{figure}[t!]
\centering
{\includegraphics[height =9.1cm]{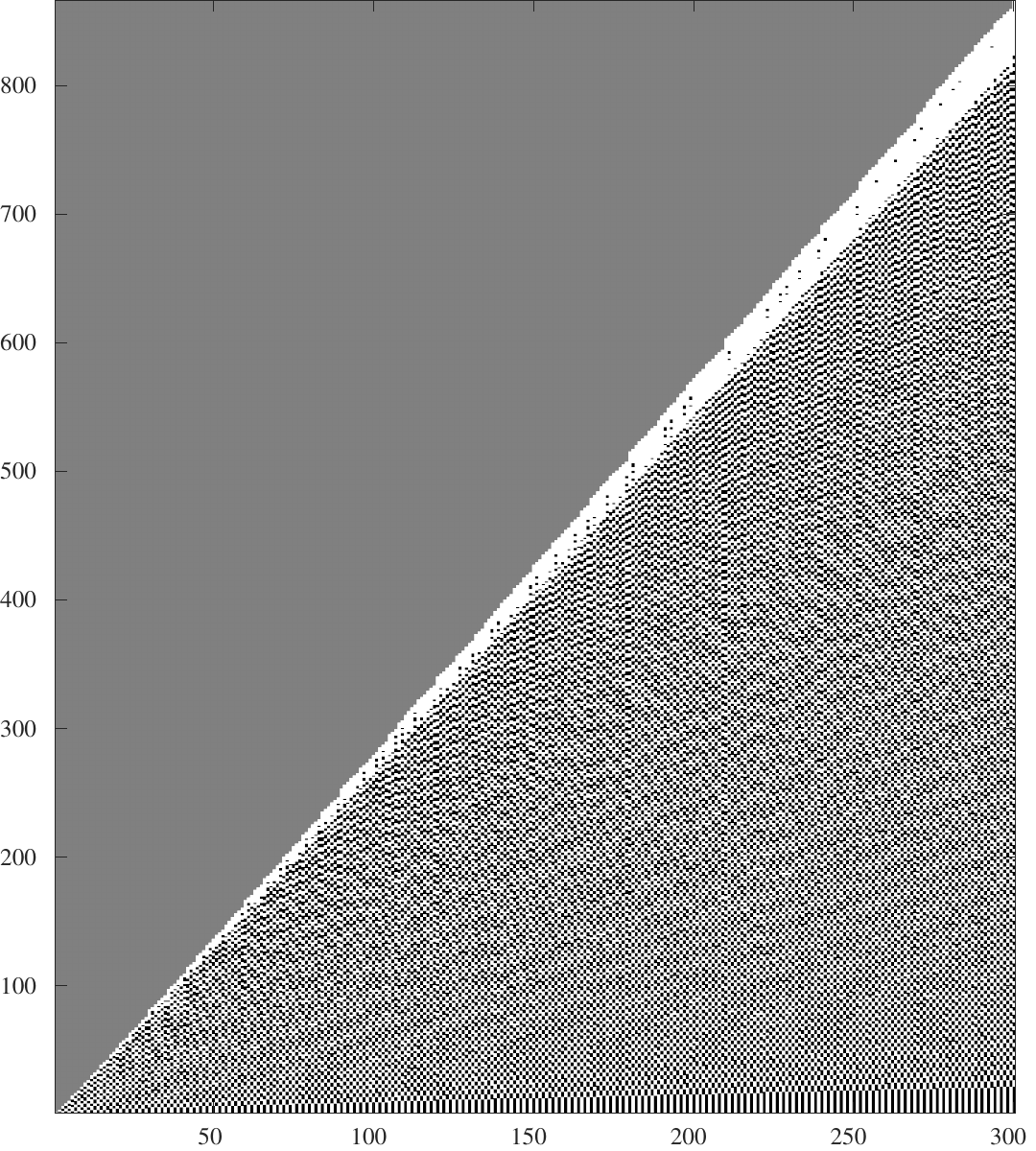}} 
{\includegraphics[height =9.1cm]{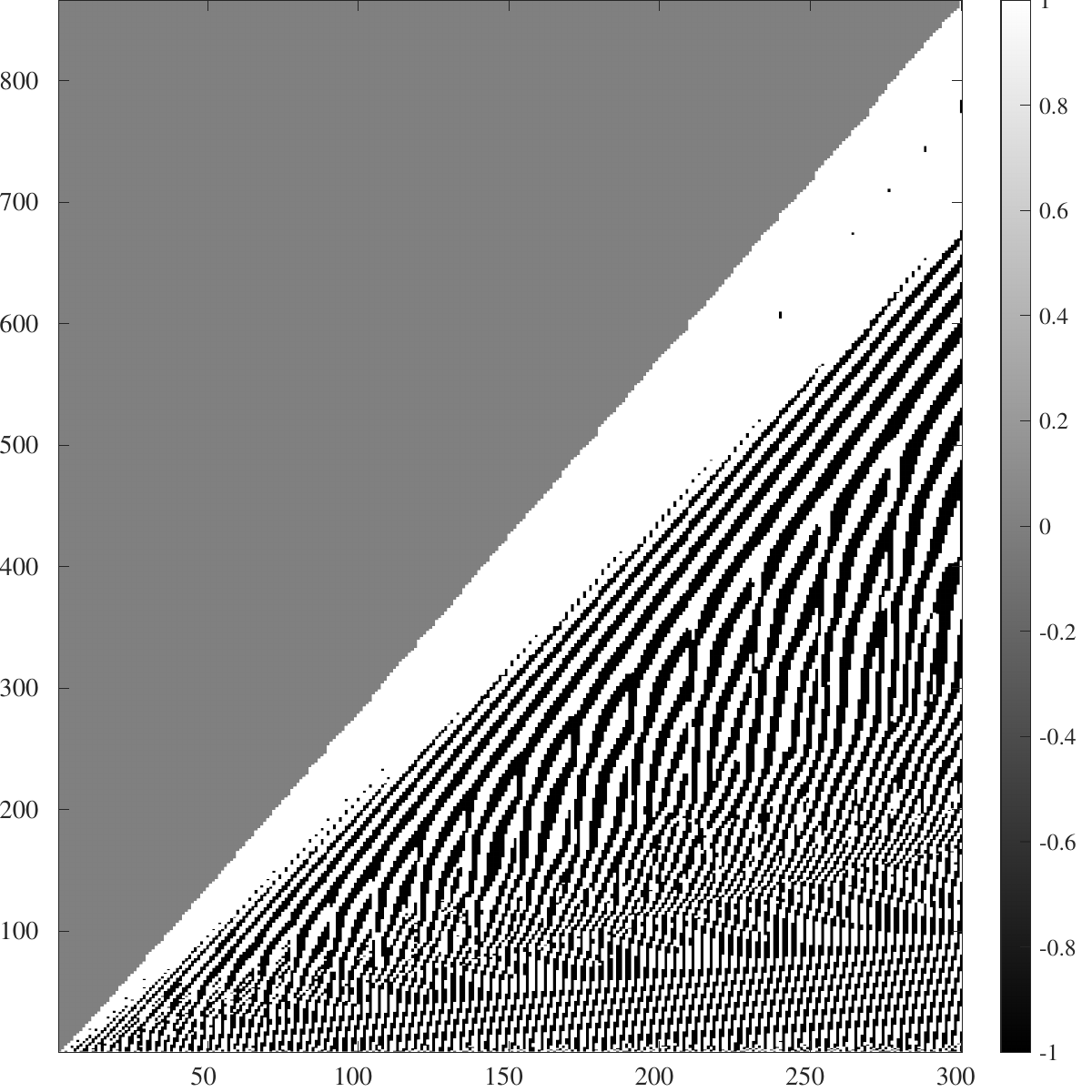}}
\caption{The sign of coefficients \(\hat b_{n,i}\) (left) and \(\widehat B_{n,i}\) (right). Here \(n=1,2,\dots,300\) is on the \(x\)-axis and \(i=0,1,\dots,M_n-2={\textnormal{deg}}(P_n)\) is on the \(y\)-axis. } 
\label{fig3}
\end{figure}

The next result provides algernative bounds on the largest real root of a polynomial. These bounds are particularly useful for reducing the computational effort required for verifying the non-negativity of \(\nu(P_n)\) and \(\nu(\widehat P_n)\).

\begin{proposition}\label{prop_bound}
Let \(P(x)=\sum\limits_{j=0}^n a_j x^j\) be a real polynomial of degree \(n\) such that \(a_j \ge 0\) for \(m \le j\le n\). For \(r>0\) denote  \(\nu_r(P):=\max\{ r, 2c(r)\}\), where
\[
c(r):=\max\limits_{1\le i \le m} \sqrt[i]{\frac{|a_{m-i}|}{b(r)}}  
\;\; {\textnormal{ and }} \; 
b(r):=\sum\limits_{j=m}^n a_j r^{j-m}. 
\]
Then  \(P(x) > 0\) for all \(x > \nu_r(P)\).  
\end{proposition}
\begin{proof}
Assume that \(z  \in [r,\infty)\) and \(P(z)=0\). Then 
\[
\sum\limits_{j=m}^n a_j z^j =  - \sum\limits_{j=0}^{m-1} a_j z^j.  
\]
Since \(z \ge r\) and \(a_j\ge 0\) for \(m\le j \le n\), we have 
\[
z^m b(r) =z^m \sum\limits_{j=m}^n a_j r^{j-m} \le \sum\limits_{j=m}^n a_j z^j =
\Big \vert - \sum\limits_{j=0}^{m-1} a_j z^j \Big \vert 
\le \sum\limits_{j=0}^{m-1} |a_j| z^j,
\]
which implies
\[
1 \le \sum\limits_{j=0}^{m-1} \frac{|a_j|}{b(r)} z^{j-m}
=\sum\limits_{i=1}^{m} \frac{|a_{m-i}|}{b(r)} z^{-i}
\le \sum\limits_{i=1}^{m} c(r)^i z^{-i} 
\le \frac{c(r) z^{-1}}{1-c(r) z^{-1}}.
\]
From this last inequality it follows that \(c(r) z^{-1} \ge 1/2\), so that \(z \le 2 c(r)\). Thus we proved that every root \(z\) of the polynomial \(P\) in the interval \([r,\infty)\) must satisfy \(z\le 2c(r)\), which is equivalent to saying that \(P\) has no roots \(z\) in the interval \(z> \max\{ r, 2c(r)\}\). This fact, together with the observation that  \(P(x) \to +\infty\) as \(x\to +\infty\), implies the statement of this result.  
\end{proof}

\begin{figure}[t!]
\centering
{\includegraphics[height =8.5cm]{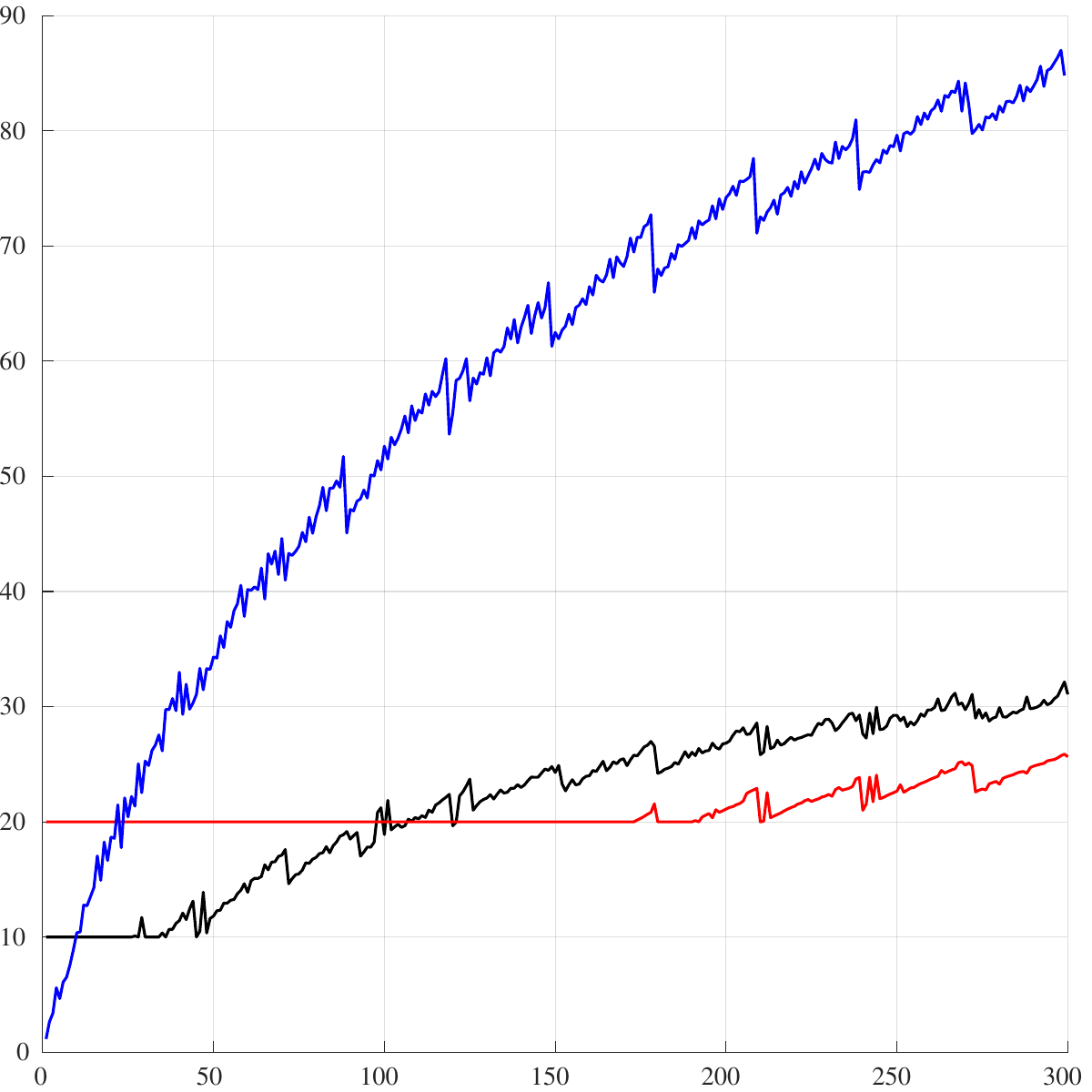}} 
{\includegraphics[height =8.5cm]{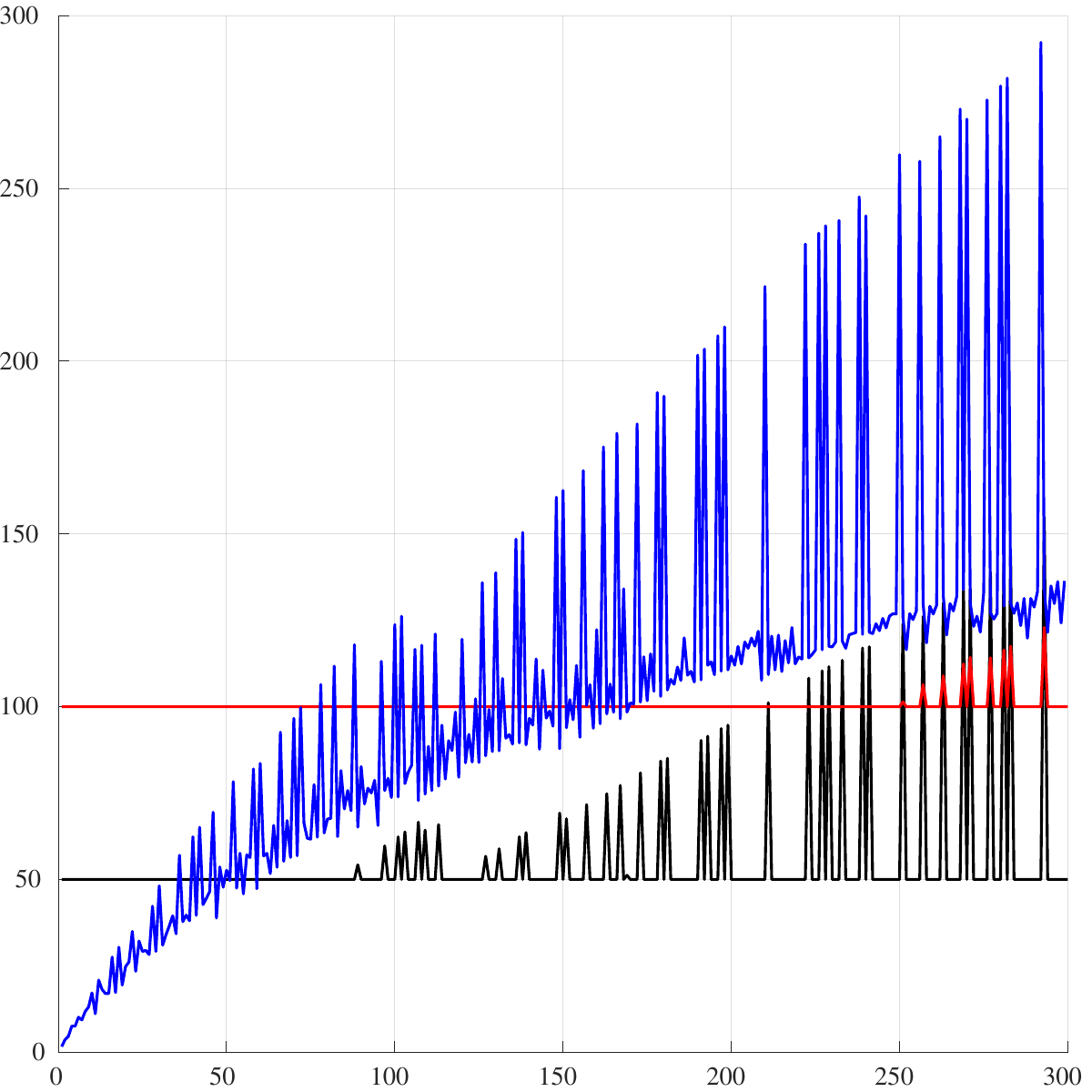}}
\caption{The left plot shows the values of \(\nu(P_n)\) (blue), \(\nu_{10}(P_n)\) (black) and \(\nu_{20}(P_n)\) (red) for \(n=1,2,\dots,300\). The right plot shows the corresponding values of 
\(\nu(\widehat P_n)\), \(\nu_{50}(\widehat P_n)\)  and \(\nu_{100}( \widehat P_n)\).} 
\label{fig4}
\end{figure}

The left plot in Figure \ref{fig4}  shows the Lagrange bounds \(\nu(P_n)\) in blue color, as well as the bounds  \(\nu_r(P_n)\) for \(r=10\) (respectively, \(r=20\)) in black (respectively, red) color. Similarly, the right plot in Figure \ref{fig4} shows the same bounds for polynomials \(\widehat P_n\), though now we take \(r=50\) and \(r=100\). The \(x\)-axis in both plots corresponds to \(n=1,2,\dots,300\).  The sharper bounds \(\nu_r(P_n)\) and \(\nu_r(\widehat P_n)\) allow for more efficient verification of the non-negativity of  \(P_n(k)\) and \(\widehat P_n(k)\). For example, to confirm that \(P_{199}(k)\ge 0\) for all \(k\ge 1\) it is enough to check numerically that \(P_{199}(k) \ge 0\) for \(k=1,2,\dots,20\). Indeed, we compute \(\nu_{20}(P_{199})\approx 20.96\) (see the red graph on the left plot in Figure \ref{fig4}), thus Proposition \ref{prop_bound} guarantees that \(P_{199}(k)>0\) for all \(k\ge 21\). 
In contrast, verifying the non-negativity of \(\widehat P_{199}(k)\) 
requires more work.  Here \(\nu_{50}(\widehat P_{199})\approx 94.67\), thus we must confirm numerically that  \(\widehat P_{199}(k)\ge 0\) for \(k=1,2,\dots,94\) to ensure  non-negativity  for all \(k\ge 1\). This difference arises because \(P_{199}\) has 53 positive leading coefficients, whereas \(\widehat P_{199}\) has only 11 (this is the consequence of the sporadic black dots on the right graph in Figure \ref{fig3}). As a result,  \(\nu_{r}(\widehat P_{199})\) is significantly larger than the corresponding bound for \(\nu_{r}(P_{199})\). 
This is not unexpected, as the non-negativity of \(\widehat P_n(k)\) for all \(n,k\ge 1\)  implies the non-negativity of \(P_n(k)\) for all \(n,k \ge 1\) (Sokal's Conjecture 3 is stronger than Conjecture 2, see page \pageref{Sokal_conjectures}). Consequently, verifying the former result is inherently more challenging.

\begin{figure}[t!]
\centering
{\includegraphics[height =8.5cm]{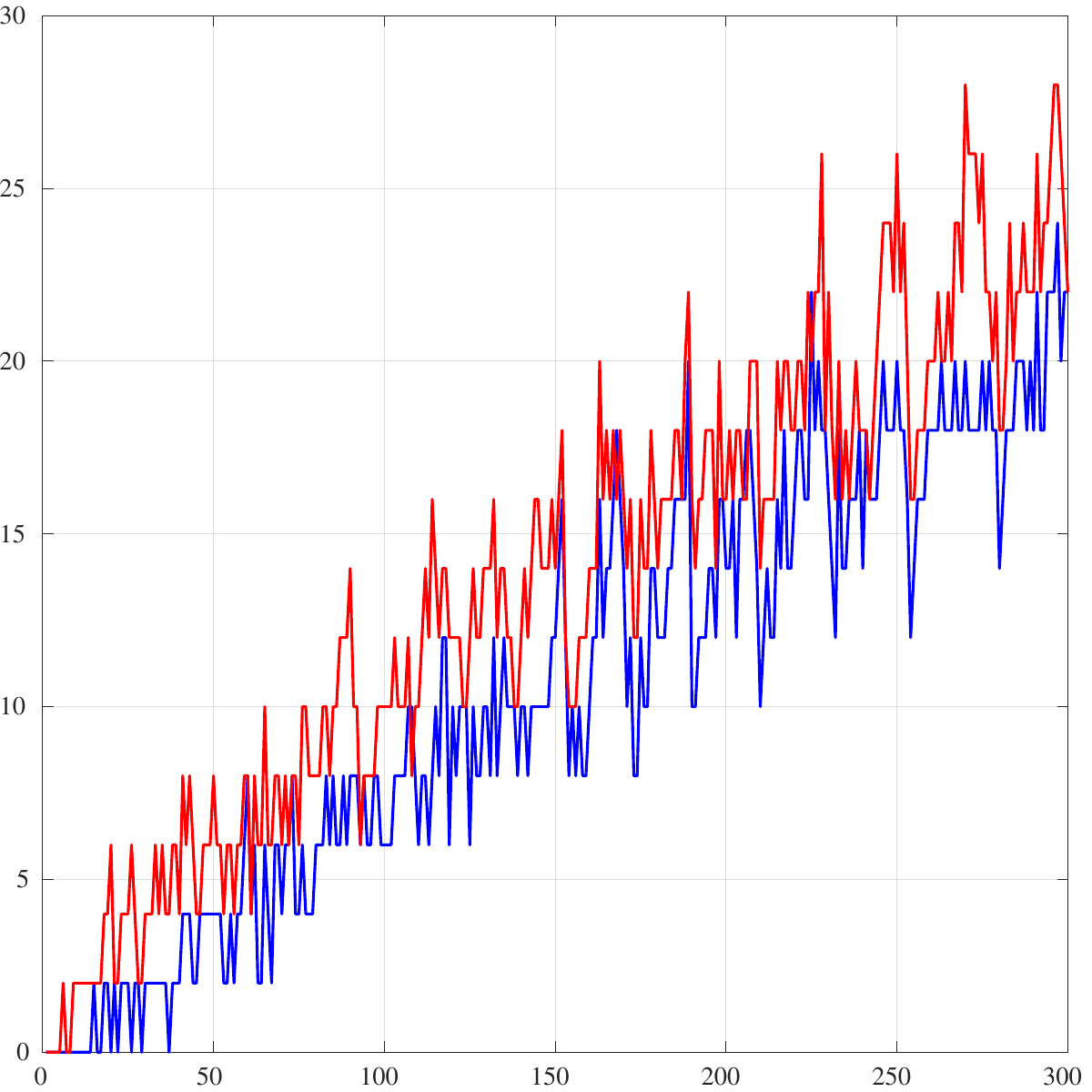}} 
{\includegraphics[height =8.5cm]{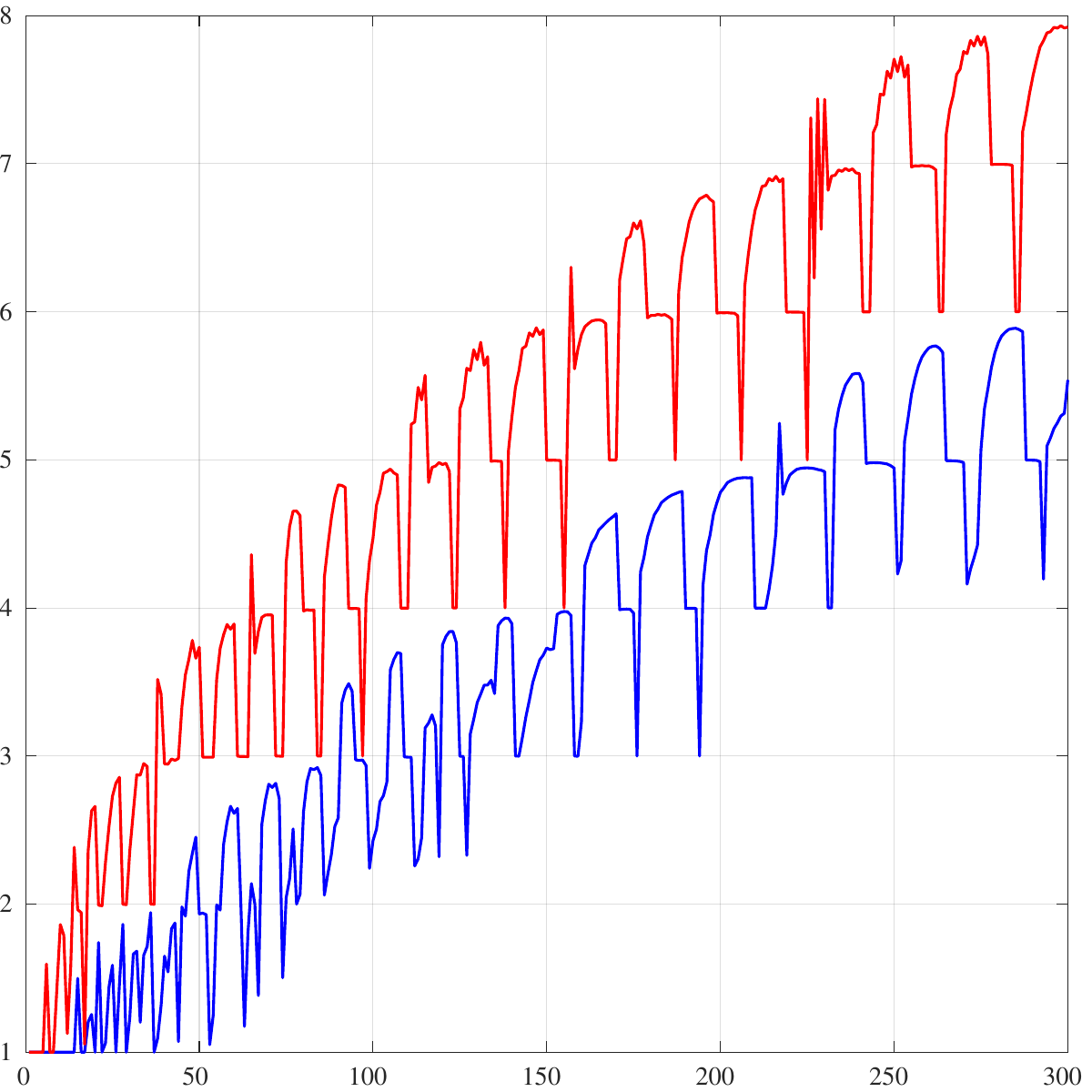}}
\caption{The left plot shows the number of real roots of  \(P_n(k)\) (blue) and \(\widehat P_n(k)\) (red) in the interval \(k \in (1,\infty)\). The right plot shows the largest positive root of \(P_n\) (blue) and \(\widehat P_n\) (red). In both plots \(n=1,2,\dots,300\) is on the \(x\)-axis.} 
\label{fig5}
\end{figure}

While \(P_n(k)\) and \(\widehat P_n(k)\) are non-negative for all \(k \in {\mathbb N}\) and $n\le 300$, a natural question is whether these polynomials remain positive for all real \(k \in (1,\infty)\)? The answer, however, is negative. Sokal \cite{Sokal_2009} observed that \(P_{15}(k)\) and \(\widehat P_6(k)\) take negative values for certain \(k>1\). We extended this analysis and computed all real roots of these polynomials that lie in the interval \(k\in (1,\infty)\). The left plot of Figure \ref{fig5} shows the number of real zeros of polynomials \(P_n(k)\) and \(\widehat P_n(k)\) on the interval \((1,\infty)\). Our computations show that starting at \(n=38\) (respectively, \(n=9\)), all polynomials \(P_n(k)\) (respectively, \(\widehat P_n(k)\)) have real roots in the interval \((1,\infty)\) and that the number of these roots tends to grow with increasing values of \(n\) (though irregularly and non-monotonically).  Interestingly, the roots tend to cluster near integers. For example,  \(\widehat P_{62}(k)\) has six real roots in the interval \((1,\infty)\) given by 
\begin{align*}
k_1&=1.000000007118404861\dots,\\
k_2&=1.000217000565656498\dots,\\
k_3&=1.143039863324097272\dots,\\
k_4&=1.999839230312833487\dots,\\
k_5&=2.484119765012083019\dots,\\
k_6&=2.991499691005777341\dots, \\
\end{align*}
with four roots close to integers. The right plot of Figure \ref{fig5} shows the graph of the largest positive root of polynomial \(P_n\) and \(\widehat P_n\) in the interval \((1,\infty)\), where  \(n=1,2\dots,300\) is on the \(x\)-axis. Again, we observe that the largest roots tend to grow with increasing \(n\) and often lie close to integers.

\section*{Acknowledgements}

The research was supported by the Natural Sciences and Engineering Research Council of Canada. The author is grateful to Alan Sokal for generously sharing unpublished results and for providing many valuable comments and suggestions on earlier drafts of this paper.

%\bibliographystyle{abbrv}
%\bibliography{references}

\appendix

\setcounter{equation}{0}
\renewcommand{\theequation}{\Alph{section}.\arabic{equation}}
\section{Polynomials \(P_n\) and \(\widehat P_n\) for \(n=1,\dots,10\)}\label{AppendixB}
Here we present the first ten polynomials \(P_n(k)\) and \(\widehat P_n(k)\). The coefficients of these polynomials for all \(n\le 300\) can be downloaded from 
\url{https://kuznetsov.mathstats.yorku.ca/code/}. 
  \small
  \begingroup
\addtolength{\jot}{-0.2em}
\begin{align*}
P_1(k)&=1, \\
P_2(k)&=3k^2 - 1, \\
 P_3(k)&=4k^5 + 8k^4 - 7k^3 - 4k^2 + 6k + 4, \\
 P_4(k)&=7k^7 + 14k^6 - 20k^5 + 34k^3 + 14k^2 - 17k - 10, \\
 P_5(k)&=6k^9 + 12k^8 - 47k^7 + 56k^6 + 134k^5 - 4k^4 - 147k^3 - 50k^2 + 54k + 28, \\
P_6(k)&=12 k^{13} + 84k^{12} + 107k^{11} - 151k^{10} + 508k^9 + 2064k^8 - 267k^7 - 4427k^6 - 2688k^5 \\
& + 3568k^4 + 4212k^3 + 2k^2 - 1512k - 504, \\
P_7(k)&=8k^{15} + 56k^{14} - 18k^{13} - 226k^{12} + 2043 k^{11} + 4597k^{10} - 4467k^9 - 13839k^8 + 175k^7 + 23793k^6 \\
&+ 14024k^5 - 14682k^4 - 16344k^3 + 132k^2 + 5136k + 1584, \\
P_8(k)&=15k^{17} + 105k^{16} + 10k^{15} - 230k^{14} + 4127 k^{13} + 6093k^{12} - 17859k^{11} - 26735k^{10} + 33965k^{9}  \\ &+ 90795k^8 + 743k^7 - 122671k^6 - 66498k^5 + 62404k^4 + 63975k^3 - 1227k^2 - 17952k - 5148, \\
P_9(k)&=13k^{20} + 117k^{19} + 65k^{18} - 343k^{17} + 8104k^{16} + 22252k^{15} - 38899k^{14} - 132601k^{13} + 93111k^{12} \\
& + 543985k^{11} + 246145k^{10} - 973159k^9 - 1119830k^8 + 569036k^7 + 1503045k^6 + 327437k^5 \\
& - 789372k^4 - 495246k^3 + 79788k^2 + 145288k + 34320, \\
P_{10}(k)&=18k^{23} + 234k^{22} + 705k^{21} - 27k^{20} + 14027k^{19} + 90809k^{18} + 14296k^{17} - 625902k^{16} - 286343k^{15} \\
&+ 3284727k^{14} + 4457207k^{13} - 6425791k^{12} - 18246084k^{11} - 1946970k^{10} + 29470120k^9 \\& + 22790446k^8 - 18143346k^7 - 29954584k^6 - 2449280k^5 + 15217542k^4 + 7441628k^3  \\&- 1797928k^2 - 2207920k - 466752,
\end{align*}
 \endgroup
 
 \vspace{-0.5cm}
   \begingroup
\addtolength{\jot}{-0.3em}
\begin{align*}
\widehat P_1(k)&=1,\\
\widehat  P_2(k)&=3k^2 - 2,\\
\widehat  P_3(k)&=4k^5 + 8k^4 - 13k^3 - 16k^2 + 9k + 10, \\
\widehat P_4(k)&=7k^7 + 14k^6 - 37k^5 - 34k^4 + 63k^3 + 52k^2 - 34k - 28,\\
\widehat P_5(k)&=6k^9 + 12k^8 - 85k^7 - 20k^6 + 263k^5 + 114k^4 - 316k^3 - 182k^2 + 126k + 84,\\
\widehat P_6(k)&=12k^{13} + 84k^{12} + 37k^{11} - 641k^{10} - 208k^9 + 3432k^8 + 2181k^7 - 8259k^6 - 9296k^5 + 5826k^4 \\ &+ 11388k^3 + 1212k^2 - 4128k - 1584, \\
\widehat P_7(k)&=8k^{15} + 56k^{14} - 134k^{13} - 1038k^{12} + 1209k^{11} + 8399k^{10} - 1437k^9 - 30289k^8 - 14762k^7 + 49794k^6 \\& + 47479k^5 - 28209k^4 - 47255k^3 - 3705k^2 + 14784k + 5148, \\
\widehat P_8(k)&=15k^{17} + 105k^{16} - 175k^{15} - 1525k^{14} + 3378k^{13} + 14270k^{12} - 18711k^{11} - 75671k^{10} + 22370k^9 \\& + 217496k^8 + 82709k^7 - 281085k^6 - 230229k^5 + 133519k^4 + 194024k^3 + 10522k^2 - 53768k - 17160,\\
\widehat P_9(k)&=13k^{20} + 117k^{19} - 193k^{18} - 2665k^{17} + 4723k^{16} + 38311k^{15} - 25510k^{14} - 286096k^{13} - 84069k^{12} \\ &+ 1078569k^{11} + 1101406k^{10} - 1794798k^9 - 3270866k^8 + 608204k^7 + 4052054k^6 + 1511634k^5 \\& - 2038568k^4 - 1608678k^3 + 145948k^2 + 454168k + 116688,\\
\widehat P_{10}(k)&=18k^{23} + 234k^{22} + 321k^{21} - 5019k^{20} - 3481k^{19} + 105761k^{18} + 124941k^{17} - 1011109k^{16} \\& - 1807088k^{15}  + 4606256k^{14} + 12818397k^{13} - 5900291k^{12} - 43470981k^{11} - 20938025k^{10} \\ &+ 65449953k^9  + 73811749k^8 - 33300902k^7 - 87108698k^6 - 17953072k^5 + 42333788k^4 + 25221880k^3 \\&- 4266288k^2 - 7093216k - 1612416.
\end{align*}

 \endgroup

\end{document}